\documentclass[sn-mathphys,Numbered]{sn-jnl}

\usepackage{graphicx}%
\usepackage{multirow}%
\usepackage{amsmath,amssymb,amsfonts}%
\usepackage{amsthm}%
\usepackage{mathrsfs}%
\usepackage[title]{appendix}%
\usepackage{xcolor}%
\usepackage{textcomp}%
\usepackage{manyfoot}%
\usepackage{booktabs}%
\usepackage{algpseudocode}%
\usepackage{listings}%
\usepackage{amsmath}
\usepackage{cases}%

\newtheorem{theorem}{Theorem}
\newtheorem{proposition}[theorem]{Proposition}%

\newtheorem{example}{Example}%
\newtheorem{lemma}[theorem]{Lemma}
\begin{document}


\title{Values and recurrence relations for integrals of powers of arctan and logarithm and associated Euler-like sums}

\author[1]{\fnm{Xiaoyu} \sur{Liu}}\email{xiaoyuliu1999@hotmail.com}
\equalcont{These authors contributed equally to this work.}
\author*[2]{\fnm{Xinhua} \sur{Xiong}}\email{xinhuax@foxmail.com}

\affil*[1,2]{\orgdiv{Department of Mathematics \& Three Gorges Mathematics Research Center}, \orgname{China Three Gorges University}, \orgaddress{\street{Daxue Road}, \city{Hubei}, \postcode{443002}, \state{Yichang}, \country{China}}}


\abstract{
In this paper, we give evaluations of integrals involving the arctan and the logarithm functions, and present several new summation identities for odd harmonic numbers and  Milgram constants.  These summation identities can be  expressed  as  finite sums of special constants such as $\pi$, the Catalan constant, the values of Riemann zeta function at the positive odd numbers and $\ln2$ etc.. Some examples are detailed  to illustrate the theorems.}

\maketitle
\section{Introduction}
\label{intro}
Recently, there has been extensive research conducted by numerous scholars on the representations of integrals that involve various special functions. For example,  the authors \cite{EG19,EG17} demonstrated a connection between the values of integrals involving the log-tangent function and the Riemann zeta function. Additionally, the authors \cite{JC,CH,XYS,FP,LBJ} used integrals to evaluate different Euler sums,
while the authors \cite{SN20} used integrals involving the arctan and the logarithm functions to represent Euler-like sums as  finite closed forms of special constants.
 In \cite{SN20}, the integral
 $$
 X(a, \delta,p,q)=\int_{0}^{1} x^{a}\ln^{q}(x)(\arctan (\delta x))^{p}dx
 $$
 was considered and the results can be used to show
 \begin{align*}
 X(1,1,2,2)=&\int_{0}^{1}x^{2}\ln^{2}(x)(\arctan(x))^{2}dx=
 \sum_{n\geq1}\frac{(-1)^{n+1}h_{n}}{n(n+1)^{3}}\\
 =&\frac{1}{2}\pi G+\frac{151}{64}\zeta(4)-\frac{7}{4}\zeta(3)\ln2 +\frac{1}{2}\zeta(2)\ln^{2}2
 -2Li_{4}\left(\frac{1}{2}\right)
+\frac{9}{8}\zeta(2)\\
 &+\frac{7}{2}\ln2-\frac{1}{2}\zeta(3)-\frac{3}{4}\pi
 -\frac{1}{12}\ln^{4}2,
 \\
 X(-2, 1, 5,0)=&\int_{0}^{1}\frac{1}{x^{2}}\arctan^{5}(x)dx=\frac{3}{2}
 \sum_{n\geq1}\sum_{j\geq1}\sum_{k=1}^{j}\frac{(-1)^{n+j}h_{n}h_{k}}{n(2j+1)(n+j)k}\\
 =&\frac{5}{32}\pi^{3}G-\frac{\pi^{5}}{1024}-\frac{135}{256}\zeta(2)\zeta(3)-
 \frac{7905}{1024}\zeta(5)+\frac{225}{256}\zeta(4)\ln2-\frac{15\pi}{4}\beta(4).
 \end{align*}
Where $G$ is the Catalan Constant and $\zeta(n)$ is the value of the Riemann zeta function at $z=n$, and\[H_{n}=\sum_{k=1}^{n}\frac{1}{k}, h_{n}=H_{2n}-\frac{1}{2}H_{n}=\sum_{k=1}^{n}\frac{1}{2k-1}.\]
In this note, we investigate integrals involving arctan and logarithmic functions  in a more complex form.  Define
\[I(a,p,q,r)=\int_{0}^{1}\frac{x^{a-p}\ln^{q}(x)(\arctan (x))^{p}}{(1+x^{2})^{r}}dx,\]
where $a, p, q\in  \mathbb{N}, a \geq p$. We provide recurrence relations for  values of  the integrals, and  finite closed forms are given for some special values of $a,p,q,$ and $r$.  As applications,  many new Euler-like sums involving $h_{n}$  and the Milgram constants are derived.  For example, in Section 3, we prove
\begin{align*}
\sum_{n=1}^{\infty}\frac{(-1)^{n+1}h_{n}}{n(2n+4)}=&
\sum_{n=0}^{\infty}\frac{(-1)^{n}t_{n}(1)}{2(n+1)(n+2)}
=\frac{1}{12}(1+\pi-4\ln 2),\\
\sum_{n=1}^{\infty}\frac{(-1)^{n}h_{n}}{n(2n+4)^{2}}=&
\sum_{n=0}^{\infty}\frac{(-1)^{n+1}t_{n}(1)}{(n+1)(2n+2)^{2}}\\
=&\frac{1}{288}(-26+2\pi(-19+18G+2\pi)+104\ln 2-63\zeta (3)).
\end{align*}
We introduce some notations that are used in the subsequent sections.
The Milgram constant $t_{k}(n)$ is defined in  \cite{MA} as
\[t_{k}(n)=\sum_{m_{n}=0}^{k}\frac{1}{(m_{n}+\frac{n}{2})}
\sum_{m_{n-1}=0}^{m_{n}}\frac{1}{(m_{n-1}+\frac{n-1}{2})}\cdot\cdot\cdot
\sum_{m_{2}=0}^{m_{3}}\frac{1}{(m_{2}+1)}\sum_{m_{1}=0}^{m_{2}}\frac{1}{(m_{1}+\frac{1}{2})}.\]
The Hurwitz  zeta  function is defined by
	\[ \zeta (p,a)=\sum_{n\geq0}\frac{1}{(n+a)^{p}}, (\Re p >1 ).\]
Moreover, the Hurwitz  zeta function has the following relation with the polygamma function,
$$\psi^{(n)}(z)=(-1)^{n+1}n!\sum_{k=0}^{\infty}\frac{1}{(z+k)^{n+1}}=(-1)^{n+1}n!\zeta (n+1,z).$$
For $  |z|\leq1$, the polylogarithm function is defined by \[Li_{n}(z)=\sum_{k=1}^{\infty}\frac{z^{k}}{k^{n}},n\in  N\setminus\{1\}.\]
 When $z = \frac{1+i}{2}$, from \cite{BN},
\begin{align*}
W(3):=&\Im\left(Li_{3}\left(\frac{1\pm i}{2}\right)\right)
=\sum_{n\geq1}\frac{\sin(\frac{n\pi}{4})}{2^{\frac{n}{2}}n^{3}}\\
=&\sum_{n\geq1}\frac{(-1)^{n+1}}{2^{2n}}\left(\frac{2}{(4n-3)^{3}} +\frac{2}{(4n-2)^{3}}+\frac{1}{(4n-1)^{3}}  \right).
\end{align*}
\\ The Catalan Constant $ G=\beta(2)$ is a special case of the Dirichlet beta function \cite{IM},  where
\[\beta(z)=\sum_{n=1}^{\infty}\frac{(-1)^{n+1}}{(2n-1)^{z}},\Re(z)>0.\]
 From \cite{KK},  we note that
\[\psi^{(2q-1)}\left(\frac{1}{4}\right)-\psi^{(2q-1)}\left(\frac{3}{4}\right)=2^{4q}(2q-1)!\beta(2q).\]

\section{ Analysis of  $ I(a,p,q,r)$}

\begin{theorem}\label{CX1.1}
Let $ a,p,q,r $ be non-negative integers and $a\geq1$ , then
\begin{align}\label{JF1}
I(a,p,q,r)=&\int_{0}^{1}\frac{x^{a}\ln^{q}(x)(\frac{\arctan x}{x})^{p}}{(1+x^{2})^{r}}dx
=\frac{(-1)^{q}q!\Gamma(p+1)}{2^{p}}\nonumber\\
&\times\sum_{k=0}^{\infty}\frac{(-1)^{k}\Gamma(r+k)}{\Gamma(k+1)\Gamma(r)}
\sum_{n=0}^{\infty}\frac{(-1)^{n}t_{n}(p-1)}{(n+\frac{p}{2})(a+2k+2n+1)^{q+1}}.
\end{align}
\end{theorem}
\begin{proof}From \cite{MA},
\[\left(\frac{\arctan x}{x}\right)^{p}=\frac{\Gamma(p+1)}{2^{p}}
\sum_{n=0}^{\infty}\frac{(-x^{2})^{n}t_{n}(p-1)}{n+\frac{p}{2}},\]
\vspace{-0.3cm}
\[t_{n}(p-1)=\sum_{k_{p-1}=0}^{n}\frac{1}{k_{p-1}+\frac{p-1}{2}}
\sum_{k_{p-2}=0}^{k_{p-1}}\frac{1}{k_{p-2}+\frac{p-2}{2}}\cdot\cdot\cdot
\sum_{k_{2}=0}^{k_{3}}\frac{1}{k_{2}+1}\sum_{k_{1}=0}^{k_{2}}\frac{1}{k_{1}+\frac{1}{2}},\]
and \[\frac{1}{(1+x^{2})^{r}}=\sum_{k=0}^{\infty}\frac{-r(-r-1)\cdot\cdot\cdot(-r-k+1)}{k!}x^{2k}
=\sum_{k=0}^{\infty}\frac{(-1)^{k}\Gamma(r+k)}{\Gamma(k+1)\Gamma(r)}x^{2k},\]
then
\begin{align*}
I(a,p,q,r)=&\int_{0}^{1}\frac{x^{a}\ln^{q}(x)(\frac{\arctan x}{x})^{p}}{(1+x^{2})^{r}}dx=\int_{0}^{1}\frac{\Gamma(p+1)}{2^{p}}
\sum_{k=0}^{\infty}\frac{(-1)^{k}\Gamma(r+k)}{\Gamma(k+1)\Gamma(r)}
\\
&\times\sum_{n=0}^{\infty}\frac{(-1)^{n}t_{n}(p-1)}{(n+\frac{p}{2})} x^{2n+2k+a}\ln^{q}(x)dx=\frac{(-1)^{q}q!\Gamma(p+1)}{2^{p}}
\\&\times
\sum_{k=0}^{\infty}\frac{(-1)^{k}\Gamma(r+k)}{\Gamma(k+1)\Gamma(r)}
\sum_{n=0}^{\infty}\frac{(-1)^{n}t_{n}(p-1)}{(n+\frac{p}{2})(2n+2k+a+1)^{q+1}}.
\end{align*}\end{proof}
In the following of this section, we first deal with the cases of $r = 0$ and $1$, showing  closed forms for $I(a,p,q,r)$ with small $p$ and $q$.  When $r > 1$, the recurrence relations of the integral and its initial values are given.

\begin{lemma}\label{PTH1}
For $ k\in \mathbb{ N}, $ then
\begin{align}
 I(2k+2,2,0,0)=&\int_{0}^{1}x^{2k}(\arctan (x))^2dx=\frac{1}{2(2k+1)}\left[ \frac{\pi^{2}}{8}-\sum_{j=1}^{k}(-1)^{j+1}\right.\nonumber
\\&\times\frac{\pi+H_{\frac{2k-2j-1}{4}}-H_{\frac{2k-2j+1}{4}}}{2k-2j+2}
\left.-(-1)^{k}2G+\frac{(-1)^{k}\pi\ln2}{2}\right]\label{EQPTH1.1},
\\
I(2k+3,2,0,0)=&\int_{0}^{1}x^{2k+1}(\arctan (x))^2dx =\frac{1}{4(k+1)}\left[ \frac{\pi^{2}}{8}-\sum_{j=1}^{k+1}(-1)^{j+1}\right.\nonumber \\
&\times\left.\frac{\pi+H_{\frac{k-j}{2}}-H_{\frac{k-j+1}{2}}}{2k-2j+3}
-\frac{\pi^{2}(-1)^{k+1}}{8}\right]\label{EQPTH1.2}.
\end{align}
\end{lemma}
\begin{proof}
Integrating by parts,
\begin{align*}
I(a+2,2,0,0)=&\int_{0}^{1}(\arctan (x))^{2}d\left(\frac{x^{a+1}}{a+1}\right)
\\
=&\frac{\pi^{2}}{16(a+1)}-\frac{2}{a+1}\int_{0}^{1}
\frac{x^{a+1}}{1+x^{2}}\arctan (x)dx.
\end{align*}
Since
\begin{align}\label{DXSCF}
\frac{x^{a+1}}{1+x^{2}}=
\begin{cases}
\sum_{j=1}^{k}(-1)^{j+1}x^{a+1-2j} +(-1)^{k}\frac{x}{1+x^{2}}, & a=2k,\\
 \sum_{j=1}^{k+1}(-1)^{j+1}x^{a+1-2j} +(-1)^{k+1}\frac{1}{1+x^{2}}, & a=2k+1,
\end{cases}
\end{align}
then
\begin{align*}\int_{0}^{1}\frac{x^{2k+1}\arctan (x)}{1+x^2}dx=&\int_{0}^{1}\sum_{j=1}^{k}(-1)^{j+1}x^{2k+1-2j}\arctan (x)dx\\&+\int_{0}^{1}\frac{(-1)^{k}x\arctan (x)}{1+x^{2}}dx,
	\\\int_{0}^{1}\frac{x^{2k+2}\arctan (x)}{1+x^2}dx=&\int_{0}^{1}\sum_{j=1}^{k+1}(-1)^{j+1}x^{2k+2-2j}\arctan (x)dx\\
&+\int_{0}^{1}\frac{(-1)^{k+1}\arctan (x)}{1+x^{2}}dx.\end{align*}
By the identity in \cite{SN19}
\[\int_{0}^{1}x^{m}\arctan (x)dx=\frac{\pi+H_{\frac{m-2}{4}}-H_{\frac{m}{4}}}{4(1+m)},\]
and
\[\int_{0}^{1}\frac{x\arctan (x)}{1+x^{2}}dx=\frac{1}{8}(4G-\pi\ln2),
	\int_{0}^{1}\frac{\arctan (x)}{1+x^{2}}dx=\frac{\pi^{2}}{32},\]
\eqref{EQPTH1.1} and \eqref{EQPTH1.2} are proved.
\end{proof}

\begin{lemma}\label{PTH2}
Let $ k\in\mathbb{N}, $ then
\begin{align}
I(2k+3,3,0,0)=&\int_{0}^{1}x^{2k}(\arctan (x))^3dx 	=\frac{\pi^{3}}{64(2k+1)}-\frac{3}{2k+1}\sum_{j=1}^{k}\frac{(-1)^{j+1}}{4(k+1-j)}
\nonumber\end{align}\begin{align}
&\times\left[ \frac{\pi^{2}}{8}-\sum_{j'=1}^{k+1-j}(-1)^{j'+1} \frac{\pi+H_{\frac{k-j-j'}{2}}-H_{\frac{k-j-j'+1}{2}}}{2k-2j-2j'+3}-(-1)^{k+1-j}\frac{\pi^{2}}{8}\right] \nonumber\\
&-\frac{(-1)^{k}3}{64(2k+1)}(16G\pi-\pi^{2}\ln4-21\zeta(3))\label{EQPTH2.1},
\end{align}\vspace{-0.5cm}\begin{align}
I(2k+4,3,0,0)=&\int_{0}^{1}x^{2k+1}(\arctan (x))^3dx=\frac{\pi^{3}}{128(k+1)}-\frac{3}{2k+2}\sum_{j=1}^{k+1}\frac{(-1)^{j+1}}{2(2k+3-2j)} \nonumber\\
&\times\left[ \frac{\pi^{2}}{8}-\sum_{j'=1}^{k+1-j}(-1)^{j'+1} \frac{\pi+H_{\frac{2k-2j-2j'+1}{4}}-H_{\frac{2k-2j-2j'+3}{4}}}{2k-2j-2j'+4} \right. \nonumber
\\
&\left.-(-1)^{k+1-j}2G+\frac{(-1)^{k+1-j}\pi\ln2}{2}\right]
-\frac{(-1)^{k+1}\pi^{3}}{128(k+1)}\label{EQPTH2.2}
.
\end{align}
\end{lemma}

\begin{proof}
\begin{align*}
I(a+3,3,0,0)&=\int_{0}^{1}x^{a}(\arctan (x))^3dx \\ &=\frac{\pi^{3}}{64(a+1)}-\frac{3}{a+1}\int_{0}^{1}\frac{x^{a+1}}{1+x^{2}}(\arctan (x))^2dx,
\end{align*}
by  \eqref{DXSCF}, then
\begin{align*}
\int_{0}^{1}\frac{x^{2k+1}(\arctan (x))^{2}}{1+x^2}dx=&\int_{0}^{1}\sum_{j=1}^{k}(-1)^{j+1}x^{2k+1-2j}(\arctan (x))^{2}dx
\\&+\int_{0}^{1}\frac{(-1)^{k}x(\arctan (x))^{2}}{1+x^{2}}dx,
\\
\int_{0}^{1}\frac{x^{2k+2}(\arctan (x))^{2}}{1+x^2}dx=&\int_{0}^{1}\sum_{j=1}^{k+1}(-1)^{j+1}x^{2k+2-2j}(\arctan (x))^{2}dx\\
&+\int_{0}^{1}\frac{(-1)^{k+1}(\arctan (x))^{2}}{1+x^{2}}dx.\end{align*}
By Lemma \ref{PTH1} and
\begin{align*}
\int_{0}^{1}\frac{x(\arctan (x))^{2}}{1+x^{2}}dx=&\frac{1}{64}(16G\pi-\pi^{2}\ln4-21\zeta(3)),\\
\int_{0}^{1}\frac{(\arctan (x))^{2}}{1+x^{2}}dx=&\frac{\pi^3}{192},
\end{align*}
  \eqref{EQPTH2.1} and \eqref{EQPTH2.2} are proved.
\end{proof}

\begin{lemma}\label{PTH4}Let $ k\in \mathbb{N},$ then
\begin{align}\label{EQPTH4.1}
I(2k+2,2,1,0)=&\int_{0}^{1}x^{2k}(\arctan (x))^{2}\ln (x)dx=\frac{1}{2k+1}\left\{-\frac{\pi^2}{16}+ \frac{k}{2k+1}\left( \frac{\pi^2}{8}\right. \right. \nonumber
\\
 &-\sum_{j=1}^{k}(-1)^{j+1}\frac{\pi+H_{\frac{2k-2j-1}{4}}-H_{\frac{2k-2j+1}{4}}}{2+2k-2j} -(-1)^{k}2G\nonumber\\
  &\left.+(-1)^{k}\frac{\pi \ln2}{2}\right)
-2\left[\sum_{j=1}^{k}\frac{(-1)^{j+1}}{64(k+1-j)^{2}}\left( -4\pi+4\right.\right.\nonumber\\
&\times\psi^{(0)}\left(\frac{2k+5-2j}{4}\right)-4\psi^{(0)}\left(\frac{2k+3-2j}{4}\right)
-2(k+1-j)\nonumber\\
&\times\left.\psi^{(1)}\left(\frac{2k+5-2j}{4}\right) +2(k+1-j)\psi^{(1)}\left(\frac{2k+3-2j}{4}\right)\right) \nonumber\\
&+\left.\frac{(-1)^{k}}{96}\left(3\pi^{3} +6\pi(\ln2)^2-192W(3)\right)\right]+2\left[ \sum_{j=1}^{k}(-1)^{j+1}\right.\nonumber
\\ 	&\times\left.\left.\frac{\pi+H_{\frac{2k-1-2j}{4}}-H_{\frac{2k+1-2j}{4}}}{8(k+1-j)}
+\frac{(-1)^{k}}{8}(4G-\pi\ln2)\right]\right\},
\\
\label{EQPTH4.2}
	I(2k+3,2,1,0)=&\int_{0}^{1}x^{2k+1}(\arctan (x))^{2}\ln (x)dx=\frac{1}{2(k+1)}\left\{-\frac{\pi^2}{16}  \right.\nonumber\\
 &+ \frac{2k+1}{4(k+1)}\left( \frac{\pi^2}{8}-\sum_{j=1}^{k+1}(-1)^{j+1}\frac{\pi+H_{\frac{k-j}{2}}-H_{\frac{k-j+1}{2}}}{2k-2j+3} \right.
 \nonumber \\
	&\left.-\frac{(-1)^{k+1}\pi^{2}}{8}\right)
-2\left[\sum_{j=1}^{k+1}\frac{(-1)^{j+1}}{16(2k+3-2j)^{2}} \right.\nonumber
\\
&\times\left( -4\pi+4\psi^{(0)}\left(\frac{k+3-j}{2}\right)-4\psi^{(0)}\left(\frac{k+2-j}{2}\right)
\right.\nonumber
\\&-(2k+3-2j)\psi^{(1)}\left(\frac{k+3-j}{2}\right)
+(2k+3-2j)\nonumber\\
&\times\left.\left.\psi^{(1)}\left(\frac{k+2-j}{2}\right)\right)+\frac{(-1)^{k+1}}{16}(-4G\pi+7\zeta(3))\right]\nonumber\\
&\left.+2\left[ \sum_{j=1}^{k+1}(-1)^{j+1}\frac{\pi+H_{\frac{k-j}{2}}-H_{\frac{k+1-j}{2}}}{4(2k+3-2j)}
+(-1)^{k+1}\frac{\pi^2}{32}\right]\right\}.
	\end{align}
\end{lemma}
\begin{proof}
Integrating by parts, then
\begin{align}\label{LO1}
(a+1)\int_{0}^{1}x^{a}(\arctan (x))^2\ln (x)dx=&-\frac{\pi^2}{16}+a\int_{0}^{1}x^a(\arctan (x))^2dx\nonumber
\end{align}
\begin{align}
-2\int_{0}^{1}\frac{x^{a+1}\ln (x) (\arctan (x))}{1+x^{2}}dx+2\int_{0}^{1}\frac{x^{a+1}\arctan (x)}{1+x^{2}}dx,
\end{align}
 for the last two integrals in \eqref{LO1}, using  \eqref{DXSCF},  then
	\begin{align*}
\int_{0}^{1}\frac{x^{2k+1}\arctan (x)\ln (x)}{1+x^2}dx=&\int_{0}^{1}\sum_{j=1}^{k}(-1)^{j+1}x^{2k+1-2j}\arctan (x)\ln (x)dx\\
&+\int_{0}^{1}\frac{(-1)^{k}x\arctan (x)\ln (x)}{1+x^{2}}dx,
\\
	\int_{0}^{1}\frac{x^{2k+2}\arctan (x)\ln (x)}{1+x^2}dx=&\int_{0}^{1}\sum_{j=1}^{k+1}(-1)^{j+1}x^{2k+2-2j}\arctan (x)\ln (x)dx
\\&+\int_{0}^{1}\frac{(-1)^{k+1}\arctan (x)\ln (x)}{1+x^{2}}dx.
\end{align*}
	At the same time, by Lemma \ref{PTH1} and
\begin{align*}\int_{0}^{1}x^{m}\ln (x)\arctan (x)dx=&\frac{1}{16(1+m)^{2}}\left[ -4\pi+4\psi^{(0)}\left(\frac{4+m}{4}\right)\right.\\	&-4\psi^{(0)}\left(\frac{2+m}{4}\right)-(1+m)\psi^{(1)}\left(\frac{4+m}{4}\right)
\\&\left.+(1+m)\psi^{(1)}\left(\frac{2+m}{4}\right)\right] ,\\
	\int_{0}^{1}\frac{x\ln (x)\arctan (x)}{1+x^{2}}dx=&\frac{1}{96}\left(3\pi^{3}+6\pi(\ln2)^2-192W(3)\right),\\
	\int_{0}^{1}\frac{\arctan (x)\ln (x)}{1+x^{2}}dx=&\frac{1}{16}(-4G\pi+7\zeta(3)),
\end{align*}
 Lemma \ref{PTH4} is proved.
\end{proof}

\begin{lemma}\label{PTH5}Let $ k\in \mathbb{N} ,$ then
\begin{align}\label{EQPTH5}
I(2k+2,2,1,1)=&\int_{0}^{1}\frac{x^{2k}\ln (x) (\arctan (x))^{2}}{1+x^{2}}dx=\sum_{j=1}^{k}\frac{(-1)^{j+1}}{2k-2j+1}\left\{ -\frac{\pi^{2}}{16}\right.\nonumber\\
&+\frac{k-j}{2k-2j+1}\left(\frac{\pi^{2}}{8}-(-1)^{k-j}2G+\frac{(-1)^{k-j}}{2}\pi\ln2
\right.\nonumber\\
&-\left.\sum_{j'=1}^{k-j}(-1)^{j'+1}
\frac{\pi+H_{\frac{2k-2j-2j'-1}{4}}-H_{\frac{2k-2j-2j'+1}{4}}}{2k+2-2j-2j'} \right) \nonumber\\
&-2\left[\sum_{j'=1}^{k-j}\frac{(-1)^{j'+1}}{64(k+1-j-j')^{2}}
\left(-4\pi+4\psi^{(0)}\left(\frac{2k+5-2j-2j'}{4}\right)\right.\right.\nonumber
\end{align}
\begin{align}
&-4\psi^{(0)}\left(\frac{2k+3-2j-2j'}{4}\right)
-2(k+1-j-j')\psi^{(1)}\left(\frac{2k+5-2j-2j'}{4}\right)\nonumber\\
&+\left.2(k+1-j-j')\psi^{(1)}\left(\frac{2k+3-2j-2j'}{4}\right)\right)\nonumber\\
&+\left.\frac{(-1)^{k-j}}{96}(3\pi^{3}
+6\pi(\ln2)^2-192W(3))\right]\nonumber\\
&+2\left[\sum_{j'=1}^{k-j}(-1)^{j'+1}\right.
\left.\times\frac{\pi+H_{\frac{2k-1-2j-2j'}{4}}
-H_{\frac{2k+1-2j-2j'}{4}}}{8(k+1-j-j')}\right.\nonumber
\\
&+\left.\left.(-1)^{k-j}\frac{4G-\pi\ln2}{8}\right]\right\}+(-1)^{k}\frac{-3G\pi^{2}+24\beta(4)}{48}.
\end{align}
\end{lemma}
\begin{proof}
Using  \eqref{DXSCF}, then
\begin{align*}
\int_{0}^{1}\frac{x^{2k}\ln (x)(\arctan (x))^{2}}{1+x^{2}}dx=&\sum_{j=1}^{k}(-1)^{j+1}\int_{0}^{1}x^{2k-2j}\ln (x)(\arctan (x))^2dx
\\
&+\int_{0}^{1}\frac{(-1)^{k}\ln (x)(\arctan (x))^2}{1+x^2}dx,\\
\int_{0}^{1}\frac{\ln (x)(\arctan (x))^{2}}{1+x^{2}}dx=&\frac{-G\pi^{2}}{16}+\frac{\beta(4)}{2},
\end{align*}
by Lemma \ref{PTH4}, and Lemma \ref{PTH5} is proved.
\end{proof}

\begin{lemma}\label{PTH6}
Let $ k\in \mathbb{N},$ then
\begin{align}\label{EQPTH6}
I(2k+4,3,1,0)=&\int_{0}^{1}x^{2k+1}(\arctan (x))^{3}\ln (x)dx=-\frac{\pi^{3}}{128(k+1)}\nonumber\\
&+\frac{(2k+1)\pi^{3}}{256(k+1)^{2}}-\frac{3(2k+1)}{4(k+1)^{2}}\sum_{j=1}^{k+1}\frac{(-1)^{j+1}}{2(2k+3-2j)}\left(  \frac{\pi^{2}}{8}\right.\nonumber
\\
&-\sum_{j'=1}^{k+1-j}(-1)^{j'+1}
\frac{\pi+H_{\frac{2k-2j-2j'+1}{4}}-H_{\frac{2k-2j-2j'+3}{4}}}{2k-2j-2j'+4}\nonumber\\
& -(-1)^{k+1-j}2G\left.+\frac{(-1)^{k+1-j}\pi\ln2}{2}\right)-\frac{(2k+1)\pi^{3}}{256(k+1)^{2}}
\nonumber\\& \times(-1)^{k+1}-\frac{3}{2(k+1)}\sum_{j=1}^{k+1}\frac{(-1)^{j+1}}{2k+3-2j}\left[-\frac{\pi^2}{16} \right.\nonumber
\\
&+ \frac{k+1-j}{2k+3-2j}\left( \frac{\pi^2}{8}-(-1)^{k+1-j}2G+\frac{(-1)^{k+1-j}\pi \log2}{2} \right. \nonumber\end{align}
\begin{align}
& \left.-\sum_{j'=1}^{k+1-j}(-1)^{j'+1}
\frac{\pi+H_{\frac{2k-2j-2j'+1}{4}}-H_{\frac{2k-2j-2j'+3}{4}}}{2(k-j-j'+2)}\right)
-2\sum_{j'=1}^{k+1-j}
\frac{(-1)^{j'+1}}{64(k+2-j-j')^{2}}\nonumber
\\
&\times \left(-4\pi \right.+4\psi^{(0)}\left(\frac{2k+7-2j-2j'}{4}\right)
-4\psi^{(0)}\left(\frac{2k+5-2j-2j'}{4}\right)\nonumber\\
&-2(k+2-j-j')\left(\psi^{(1)}\left(\frac{2k+7-2j-2j'}{4}\right)\right.
-\left.\left.\psi^{(1)}\left(\frac{2k+5-2j-2j'}{4}\right)\right)\right)\nonumber\\
&
+\frac{(-1)^{k+1-j}}{48}\times(3\pi^{3}
+6\pi(\ln2)^2-192W(3))+2\left(  \sum_{j'=1}^{k+1-j}(-1)^{j'+1}\right.\nonumber
\\
&\times
\frac{\pi+H_{\frac{2k+1-2j-2j'}{4}}-H_{\frac{2k+3-2j-2j'}{4}}}{8(k+2-j-j')}
+(-1)^{k+1-j}\left.\left.\frac{(4G-\pi\ln2)}{8}\right)\right]\nonumber\\
&-(-1)^{k+1}
\frac{-3G\pi^{2}+24\beta(4)}{32(k+1)}
+\frac{3}{2(k+1)}\left[\sum_{j=1}^{k+1}\frac{(-1)^{j+1}}{2(2k+3-2j)}\right.
\left(\frac{\pi^{2}}{8}\right.\nonumber\\
&-(-1)^{k+1-j}2G
-\sum_{j'=1}^{k+1-j}(-1)^{j'+1}
\frac{\pi+H_{\frac{2k+1-2j-2j'}{4}}-H_{\frac{2k+3-2j-2j'}{4}}}{2(k+2-j-j')}
\nonumber
\\
&+\left.\left.\frac{(-1)^{k+1-j}}{2}\pi\ln2\right)
+(-1)^{k+1}\frac{\pi^{3}}{192}\right].
\end{align}
\end{lemma}
\begin{proof}
Integrating by parts, then
\begin{align*}
	(a+1)\int_{0}^{1}x^{a}\ln (x)(\arctan(x))^3dx=&a\int_{0}^{1}x^{a}(\arctan(x))^3dx-\frac{\pi^3}{64}
\\&+3\int_{0}^{1}\frac{x^{a+1}(\arctan(x))^2}{1+x^2}dx
\\
&-3\int_{0}^{1}\frac{x^{a+1}\ln (x)(\arctan(x))^{2}}{1+x^{2}}dx,\end{align*}
by Lemma \ref{PTH1}, Lemma \ref{PTH2} and Lemma \ref{PTH4}, Lemma \ref{PTH6} is proved.
\end{proof}
\begin{lemma}\label{PTH7}Let $ k\in \mathbb{N}, $ then
\begin{align}\label{EQPTH7.1}
I(2k+1,1,2,1)=&\int_{0}^{1}\frac{x^{2k}\arctan (x)(\ln (x))^{2}}{1+x^2}dx
=\sum_{j=1}^{k}\frac{(-1)^{j+1}}{64(1+2k-2j)^{3}}\nonumber\\
&\times\left(32\pi-32\psi^{(0)}\left(\frac{k+2-j}{2}\right)+32\psi^{(0)}\left(\frac{1+k-j}{2}\right)
\right.\nonumber\\
&+8(1+2k-2j)\psi^{(1)}\left(\frac{k+2-j}{2}\right)\nonumber\\
&-8(1+2k-2j)\psi^{(1)}\left(\frac{1+k-j}{2}\right)-(1+2k-2j)^{2}\nonumber
\end{align}
\begin{align}
&\left.\times\left(\psi^{(2)}\left(\frac{k+2-j}{2}\right)
-\psi^{(2)}\left(\frac{1+k-j}{2}\right)\right)\right)
+(-1)^{k}\left[\frac{1}{2}\left(\frac{11}{4}\zeta(4)\right.\right.
\nonumber\\&-\frac{7}{4}\zeta(3)\ln2+\frac{1}{2}\zeta(2)\ln^{2}2
-\left.\left.\frac{1}{12}\ln^{4}2-2Li_{4}\left(\frac{1}{2}\right)\right)-\frac{25}{128}\zeta(4)\right],
\end{align}\vspace{-0.5cm}
\begin{align}
\label{EQPTH7.2}
I(2k+2,1,2,1)=&\int_{0}^{1}\frac{x^{2k+1}\arctan (x)(\ln (x))^{2}}{1+x^2}dx
=\sum_{j=1}^{k}\frac{(-1)^{j+1}}{512(1+k-j)^{3}}\nonumber
\\
&\times\left(32\pi-32\psi^{(0)}\left(\frac{2k+5-2j}{4}\right)+32\psi^{(0)}\left(\frac{3+2k-2j}{4}\right)
\right.\nonumber\\
&+16(1+k-j)\psi^{(1)}\left(\frac{2k+5-2j}{4}\right)-16(1+k-j)\nonumber\\
&\times\psi^{(1)}\left(\frac{3+2k-2j}{4}\right)
-4(1+k-j)^{2}\psi^{(2)}\left(\frac{2k+5-2j}{4}\right)\nonumber\\
&
+\left.4(1+k-j)^{2}\psi^{(2)}\left(\frac{3+2k-2j}{4}\right)\right)\nonumber
\\
&
+(-1)^{k}\left(\frac{7\pi}{64}\zeta(3)+\beta(4)-\frac{\pi^{3}}{16}\ln2\right).
\end{align}
\end{lemma}
\begin{proof}
By  \eqref{DXSCF}, then
\begin{align*}
\int_{0}^{1}\frac{x^{2k}\arctan (x)(\ln (x))^{2}}{1+x^2}dx=&\int_{0}^{1}\sum_{j=1}^{k}(-1)^{j+1}x^{2k-2j}\arctan (x)(\ln (x))^{2}dx\\
&+\int_{0}^{1}\frac{(-1)^{k}\arctan (x)(\ln (x))^{2}}{1+x^2}dx,
\\
\int_{0}^{1}\frac{x^{2k+1}\arctan (x)(\ln (x))^{2}}{1+x^2}dx=&\int_{0}^{1}\sum_{j=1}^{k}(-1)^{j+1}x^{2k+1-2j}\arctan (x)(\ln (x))^{2}dx\\
&+\int_{0}^{1}\frac{(-1)^{k}x\arctan (x)(\ln (x))^{2}}{1+x^2}dx.
\end{align*}
Moreover, by the following identities in \cite{SF22}
\begin{align*}
\int_{0}^{1}x^{m}\arctan (x)(\ln (x))^{2}dx=&\frac{1}{64(1+m)^{3}}\left(32\pi-32\psi^{(0)}\left(\frac{m+4}{4}\right)\right.
\\&+32\psi^{(0)}\left(\frac{2+m}{4}\right)+8(1+m)\psi^{(1)}\left(\frac{m+4}{4}\right)
\\&-8(1+m)\psi^{(1)}\left(\frac{2+m}{4}\right)-(1+m)^{2}
\\&\times\psi^{(2)}\left(\frac{m+4}{4}\right)
\left.+(1+m)^{2}\psi^{(2)}\left(\frac{2+m}{4}\right)\right),\\
\int_{0}^{1}\frac{\arctan (x)(\ln (x))^{2}}{1+x^2}dx=&\frac{1}{2}\left(\frac{11}{4}\zeta(4)-\frac{7}{4}\zeta(3)\ln2
+\frac{1}{2}\zeta(2)\ln^{2}2\right.\\
&\left.-\frac{1}{12}\ln^{4}2-2Li_{4}\left(\frac{1}{2}\right)\right)-\frac{25}{128}\zeta(4),
\\
\int_{0}^{1}\frac{x\arctan (x)(\ln (x))^{2}}{1+x^2}dx=&\frac{7\pi}{64}\zeta(3)+\beta(4)-\frac{\pi^{3}}{16}\ln2,
\end{align*}
and Lemma \ref{PTH7} is proved.
\end{proof}

\begin{lemma}\label{PTH8}Let $k\in  \mathbb{N},$ then
\begin{align}\label{EQPTH8.1}
I(2k+2,2,2,0)=&\int_{0}^{1}x^{2k}(\arctan (x))^{2}(\ln (x))^{2}dx=\frac{-2}{2k+1}\left[\sum_{j=1}^{k}\frac{(-1)^{j+1}}{(k+1-j)^{3}}\right.\nonumber\\
&\times\frac{1}{512}\left(32\pi
-32\psi^{(0)}\left(\frac{2k+5-2j}{4}\right)+32\psi^{(0)}\left(\frac{2k+3-2j}{4}\right)
\right.\nonumber\\ &+16(k+1-j)\left(\psi^{(1)}\left(\frac{2k+5-2j}{4}\right)
-\psi^{(1)}\left(\frac{2k+3-2j}{4}\right)\right)
\nonumber\\
&-\left.4(k+1-j)^{2}\left(\psi^{(2)}\left(\frac{2k+5-2j}{4}\right)
-\psi^{(2)}\left(\frac{3+2k-2j}{4}\right)\right)\right)\nonumber
\\
&\left.+(-1)^{k}\left(\frac{7\pi}{64}\zeta(3)+\beta(4)-\frac{\pi^{3}}{16}\ln2\right)\right]
-\frac{2}{(2k+1)^{2}}
\left\{-\frac{\pi^{2}}{16}+\frac{k}{2k+1}\right.\nonumber
\\
&\times\left(\frac{\pi^{2}}{8}\left.-\sum_{j=1}^{k}(-1)^{j+1}
\frac{\pi+H_{\frac{2k-1-2j}{4}}-H_{\frac{2k-2j+1}{4}}}{2(k+1-j)}
-(-1)^{k}2G\right.\right.\nonumber
\\
&+(-1)^{k}\left.\frac{\pi\ln2}{2}\right)-2\left[\sum_{j=1}^{k}\frac{(-1)^{j+1}}{64(k+1-j)^{2}}
\left(4\psi^{(0)}
\left(\frac{2k+5-2j}{4}\right)\right.\right.\nonumber\\
&-4\psi^{(0)}\left(\frac{2k+3-2j}{4}\right)-4\pi-2(k+1-j)\left(\psi^{(1)}\left(\frac{2k+5-2j}{4}\right)
\right.\nonumber\\
&-\left.\left.\left.\psi^{(1)}\left(\frac{2k+3-2j}{4}\right)\right)\right)
+\frac{(-1)^{k}}{96}(3\pi^{3}
+6\pi(\ln2)^{2}-192W(3))\right]\nonumber
\\ &+2\left[\sum_{j=1}^{k}(-1)^{j+1}\frac{\pi+H_{\frac{2k-1-2j}{4}}-H_{\frac{2k+1-2j}{4}}}{8(k+1-j)}\right.
\nonumber\\&+\left.\left.\frac{(-1)^{k}}{8}(4G-\pi\ln 2)\right]\right\},
\\
\label{EQPTH8.2}
I(2k+3,2,2,0)=&\int_{0}^{1}x^{2k+1}(\arctan (x))^{2}(\ln (x))^{2}dx=-\frac{1}{k+1}\left\{\sum_{j=1}^{k+1}\frac{(-1)^{j+1}}{64(2k+3-2j)^{3}}\right.\nonumber\\
&\times\left[32\pi
-32\psi^{(0)}\left(\frac{k+3-j}{2}\right)+32\psi^{(0)}\left(\frac{k+2-j}{2}\right)\right.\nonumber
\\
&+8(2k+3-2j)\left(\psi^{(1)}\left(\frac{k+3-j}{2}\right)-\psi^{(1)}\left(\frac{k+2-j}{2}\right)\right)\nonumber\\
&\left.-(2k+3-2j)^{2}\left(\psi^{(2)}\left(\frac{k+3-j}{2}\right)
-\psi^{(2)}\left(\frac{2+k-j}{2}\right)\right)\right]\nonumber\\
&+(-1)^{k+1}\left[\frac{1}{2}\left(\frac{11}{4}\zeta(4)-\frac{7}{4}\zeta(3)\ln2
+\frac{1}{2}\zeta(2)\ln^{2}2-\frac{1}{12}\ln^{4}2\right.\right.\nonumber
\\
&-\left.\left.\left.2Li_{4}\left(\frac{1}{2}\right)\right)
-\frac{25}{128}\zeta(4)\right]\right\}-\frac{1}{2(k+1)^{2}}\left\{
-\frac{\pi^{2}}{16}+\frac{2k+1}{4(k+1)}\right.\nonumber
\\
&\times\left.\left(\frac{\pi^{2}}{8}
-\sum_{j=1}^{k+1}(-1)^{j+1}\frac{\pi+H_{\frac{k-j}{2}}
-H_{\frac{k-j+1}{2}}}{2k+3-2j}+\frac{(-1)^{k}\pi^{2}}{8}\right)\right.\nonumber
\\
&-2\left[\sum_{j=1}^{k+1}(-1)^{j+1}\frac{1}{16(2k+3-2j)^{2}}
\left(-4\pi+4\psi^{(0)}\left(\frac{k+3-j}{2}\right)\right.\right.\nonumber\\
&-4\psi^{(0)}\left(\frac{k+2-j}{2}\right)-(2k+3-2j)\psi^{(1)}\left(\frac{k+3-j}{2}\right)\nonumber\\
&+\left.\left.(2k+3-2j)\psi^{(1)}\left(\frac{k+2-j}{2}\right)\right)+\frac{(-1)^{k+1}}{16}(-4G\pi+7\zeta(3))\right]\nonumber \\
&+\left.2\left[\sum_{j=1}^{k+1}(-1)^{j+1}\frac{\pi+H_{\frac{k-j}{2}}-H_{\frac{k+1-j}{2}}}{4(2k+3-2j)}
+\frac{(-1)^{k+1}\pi^{2}}{32}\right]\right\}.
\end{align}
\end{lemma}
\begin{proof}
Integrating by parts, then
\[I(a+2,2,2,0)=\frac{-2}{a+1}\int_{0}^{1}\frac{x^{a+1}\arctan (x)(\ln (x))^{2}}{1+x^{2}}dx-\frac{2}{a+1}\int_{0}^{1}x^{a}\ln x (\arctan (x))^{2}dx,\]
by Lemma \ref{PTH4} and Lemma \ref{PTH7}, Lemma \ref{PTH8} is proved.
\end{proof}

By taking special values for the parameters $a,p,q$, we  evaluate  closed forms for $I ( a, p, q, r )$ when $r = 0$ or $r=1$. But for $r \geq 2$, we provide its structure representations.
\vspace{-0.1cm}
\begin{proposition}\label{MT1}
Let $p=0, q=1,  r>1, r, k\in  \mathbb{N}$,  then
\begin{numcases}{\int_{0}^{1}\frac{x^{a}\ln(x)}{(1+x^{2})^{r}}dx=}
l_{0}+l_{1}\ln2+l_{2}\pi^{2},& \text{$a=2k-1,$} \label{CS1.1} \\
l_{0}+l_{1}G+l_{2}\pi,& \text{$a=2k,$}\label{CS1.2}
\end{numcases}
where  $l_{i}(i=1,2,3,\cdot\cdot\cdot)$ depends on $a, r$ and belongs to $ \mathbb{Q}. $
\end{proposition}
\begin{proof}Integrating by parts, for $a\geq 2$ and $r\geq 2$,
\begin{equation}\label{DT1}
\int_{0}^{1}\frac{x^{a}\ln (x)}{(1+x^{2})^{r}}dx=\frac{a-1}{2(r-1)}\int_{0}^{1}\frac{x^{a-2}\ln x}{(1+x^{2})^{r-1}}dx+\frac{1}{2(r-1)}\int_{0}^{1}\frac{x^{a-2}}{(1+x^{2})^{r-1}}dx.
\end{equation}
By  \eqref{DT1}, we need to consider the third integral and the cases of  $a=0, r=1 $ and $a=1, r=1$: \vspace{-0.3cm}
\begin{align}
\int_{0}^{1}\frac{\ln(x)}{(1+x^{2})}dx=&-G
,\label{OS1.1}\\
\int_{0}^{1}\frac{x\ln(x)}{(1+x^{2})}dx=&-\frac{\pi^{2}}{48}\label{JS1.1}.
\end{align}
For the third integral in \eqref{DT1},  one has
\[\int_{0}^{1}\frac{x^{a}}{(1+x^{2})^{r}}dx=
\frac{a-1}{3-2r-a}\int_{0}^{1}\frac{x^{a-2}}{(1+x^{2})^{r}}dx
-\frac{2^{r-1}}{3-2r-a}.\]
When $a = 0$,
\begin{equation}\label{YY1}
\int_{0}^{1}\frac{1}{(1+x^{2})^{r}}dx=\frac{1}{2^{r}(r-1)}
+\frac{2r-3}{2(r-1)}\int_{0}^{1}\frac{1}{(1+x^{2})^{r-1}}dx,\end{equation}
and its  initial value is
\begin{equation}\label{OS1.2}
\int_{0}^{1}\frac{1}{1+x^{2}}dx=\frac{\pi}{4}.\end{equation}
When $a = 1$,
\[\int_{0}^{1}\frac{x}{(1+x^{2})^{r}}dx=\frac{2^{-1-r}(-2+2^{r})}{-1+r},\]
the initial value is\vspace{-0.3cm}
\begin{equation}\label{JS1.2}
\int_{0}^{1}\frac{x}{1+x^{2}}dx=\frac{\ln 2}{2}.\end{equation}
Then, if $a$ is  even, $I(a,0,1,r)$ depends on \eqref{OS1.1} and \eqref{OS1.2}; if $a$ is odd, $I(a,0,1,r)$ depends on \eqref{JS1.1} and \eqref{JS1.2}, and Proposition \ref{MT1} is proved.
\end{proof}

\begin{proposition}\label{MT2}
Let $p=0, q=2, a, r\geq 2, r, k\in  \mathbb{N},$
then
\begin{numcases}{\int_{0}^{1}\frac{x^{a}\ln^{2}(x)}{(1+x^{2})^{r}}dx=}
l_{1}\zeta(3)+l_{2}\pi^{2}+l_{3}\ln2+l_{4},& \text{$a=2k-1,$} \label{CS2.3} \\
l_{1}\pi^{3}+l_{2}G+l_{3}\pi+l_{4},& \text{$a=2k,$} \label{CS2.4}
\end{numcases}
where $l_{i}(i=1,2,3,\cdot\cdot\cdot)$ depends on $a, r$ and belongs to $ \mathbb{Q}.$
\end{proposition}
\begin{proof}
Integrating by parts, we get
\begin{equation}\label{DT2}
\int_{0}^{1}\frac{x^{a}\ln^{2} (x)}{(1+x^{2})^{r}}dx=\frac{a-1}{2(r-1)}\int_{0}^{1}\frac{x^{a-2}\ln^{2} (x)}{(1+x^{2})^{r-1}}dx+\frac{1}{(r-1)}\int_{0}^{1}\frac{x^{a-2}\ln(x)}{(1+x^{2})^{r-1}}dx.
\end{equation}
Since\vspace{-0.3cm}
\begin{align}
\int_{0}^{1}\frac{x\ln^{2}(x)}{(1+x^{2})}dx=&\frac{3}{16}\zeta(3)\label{CS2.1},\\
\int_{0}^{1}\frac{\ln^{2}(x)}{(1+x^{2})}dx=&\frac{\pi^{3}}{16}.\label{CS2.2}
\end{align}
By  \eqref{DT2} and Proposition \ref{MT1}, when $a$ is even, $I(a,0,2,r)$ depends on \eqref{CS2.2} and \eqref{CS1.2}; When $a$ is odd, $I(a,0,2,r)$ depends on \eqref{CS2.1} and \eqref{CS1.1}. Therefore, Proposition \ref{MT2} is proved.
\end{proof}

\begin{theorem}\label{DTTH1}   
Let $r, n, q\geq1$ be natural numbers,  then
\begin{align}
\int_{0}^{1}\frac{x^{2n-1}\ln^{q}(x)}{(1+x^{2})^{r}}dx=&
l_{0}\ln2+\sum_{i=1}^{q}l_{i}\zeta(i+1)+l_{q+1},\\
\int_{0}^{1}\frac{x^{2n}\ln^{2k-1}(x)}{(1+x^{2})^{r}}dx=&
l_{0}G+\sum_{i=1}^{k}l_{i}\pi^{2i-1}\nonumber\\
&+\sum_{j=2}^{k}c_{j}
\left(\psi^{(2j-1)}\left(\frac{1}{4}\right)
-\psi^{(2j-1)}\left(\frac{3}{4}\right)\right)+l_{k+1},\\
\int_{0}^{1}\frac{x^{2n}\ln^{2k}(x)}{(1+x^{2})^{r}}dx=&
l_{0}G+\sum_{i=1}^{k+1}l_{i}\pi^{2i-1}\nonumber\\
&+\sum_{j=2}^{k}c_{j}
\left(\psi^{(2j-1)}\left(\frac{1}{4}\right)
-\psi^{(2j-1)}\left(\frac{3}{4}\right)\right)+l_{k+2},
\end{align}
where  all rational coefficients $l_{i},c_{j}(i,j=1,2,3,\ldots)$ depend on the values of $n, r, k$.
\end{theorem}
\begin{proof}
Define
\[A( q, r )=\int_{0}^{1}\frac{x\ln^{q}(x)}{(1+x^{2})^{r}}dx,\,\,\, B( q, r )=\int_{0}^{1}\frac{\ln^{q}(x)}{(1+x^{2})^{r}}dx.\]
Integrating by parts, for $q\geq 3, r\geq 1$,
\begin{align}\label{DGGS1}
A( q, r )=\int_{0}^{1}\frac{x\ln^{q}(x)}{(1+x^{2})^{r}}dx=&\int_{0}^{1}\frac{x\ln^{q-1}(x)}{(1+x^{2})^{r}}d (x\ln(x)-x)=-\int_{0}^{1}\frac{x\ln^{q}(x)}{(1+x^{2})^{r}}dx\nonumber\\
&+\int_{0}^{1}\frac{x\ln^{q-1}(x)}{(1+x^{2})^{r}}dx
-(q-1)\int_{0}^{1}\frac{x\ln^{q-1}(x)}{(1+x^{2})^{r}}dx\nonumber\\
&+(q-1)\int_{0}^{1}\frac{x\ln^{q-2}(x)}{(1+x^{2})^{r}}dx
+2r\int_{0}^{1}\frac{x^{3}\ln^{q}(x)}{(1+x^{2})^{r+1}}dx\nonumber\\
&-2r\int_{0}^{1}\frac{x^{3}\ln^{q-1}(x)}{(1+x^{2})^{r+1}}dx,
\end{align}
the last two integrals can be converted into
\[\int_{0}^{1}\frac{x^{3}\ln^{q}(x)}{(1+x^{2})^{r+1}}dx=
\int_{0}^{1}\frac{x\ln^{q}(x)}{(1+x^{2})^{r}}dx+
\int_{0}^{1}\frac{x\ln^{q}(x)}{(1+x^{2})^{r+1}}dx,\]
\[\int_{0}^{1}\frac{x^{3}\ln^{q-1}(x)}{(1+x^{2})^{r+1}}dx=
\int_{0}^{1}\frac{x\ln^{q-1}(x)}{(1+x^{2})^{r}}dx+
\int_{0}^{1}\frac{x\ln^{q-1}(x)}{(1+x^{2})^{r+1}}dx.\]
Replacing $r + 1$ by $r$,  \eqref{DGGS1} is equivalent to
\begin{align}
A(q,r)=&a_{1}A(q,r-1)+a_{2}A(q-1,r)+a_{3}A(q-1,r-1)+a_{4}A(q-2,r-1),\label{DG1}
\end{align}
where  the value of $a_{i}$ depends on $q$ and $r$.
Similarly,\vspace{-0.1cm}
\begin{align}\label{DGGS2}
B( q, r )=&\int_{0}^{1}\frac{\ln^{q}(x)}{(1+x^{2})^{r}}dx=\int_{0}^{1}\frac{\ln^{q-1}(x)}{(1+x^{2})^{r}}d (x\ln(x)-x)\nonumber
\\
=&-(q-1)\int_{0}^{1}\frac{\ln^{q-1}(x)}{(1+x^{2})^{r}}dx
+(q-1)\int_{0}^{1}\frac{x\ln^{q-2}(x)}{(1+x^{2})^{r}}dx\nonumber\\
&+2r\int_{0}^{1}\frac{x^{2}\ln^{q}(x)}{(1+x^{2})^{r+1}}dx
-2r\int_{0}^{1}\frac{x^{2}\ln^{q-1}(x)}{(1+x^{2})^{r+1}}dx,
\end{align}
and \vspace{-0.5cm}\begin{align*}\int_{0}^{1}\frac{x^{2}\ln^{q}(x)}{(1+x^{2})^{r+1}}dx=&
\int_{0}^{1}\frac{\ln^{q}(x)}{(1+x^{2})^{r}}dx+
\int_{0}^{1}\frac{\ln^{q}(x)}{(1+x^{2})^{r+1}}dx,\\
\int_{0}^{1}\frac{x^{2}\ln^{q-1}(x)}{(1+x^{2})^{r+1}}dx=&
\int_{0}^{1}\frac{\ln^{q-1}(x)}{(1+x^{2})^{r}}dx+
\int_{0}^{1}\frac{\ln^{q-1}(x)}{(1+x^{2})^{r+1}}dx.\end{align*}
Replacing $r + 1$ by $r$,  \eqref{DGGS2} is equivalent to
\begin{align}
B(q,r)=&b_{1}B(q,r-1)+b_{2}B(q-1,r)+b_{3}B(q-1,r-1)+b_{4}B(q-2,r-1),\label{DG2}
\end{align}
where the value of $b_{i}$ depends on $q$ and $r$.
According to \eqref{DG1} and \eqref{DG2}, the value of $A(q,r)$ depends on $A (q,1),$ $A (1,r),$ $A (2,r),$ and the value of $B(q,r)$ depends on $B(q,1), B(1,r), $$B(2,r).$ Since $A(1,r), A(2,r), B(1,r)$ and $B(2,r)$ have been shown in Proposition \ref{MT1} and Proposition \ref{MT2}, we only need to consider $A ( q, 1 )$ and $B(q,1)$.
\begin{align*}A ( q, 1 )=&\int_{0}^{1}\frac{x\ln^{q}(x)}{(1+x^{2})}dx=
\frac{(-1)^{q}}{2^{1+2q}}(-1+2^{q})\Gamma(1+q)\zeta(1+q),\\
B(q,1)=&\int_{0}^{1}\frac{\ln^{q}(x)}{(1+x^{2})}dx=\frac{(-1)^{q}}{4^{1+q}}
\Gamma(1+q)\left(\zeta\left(1+q,\frac{1}{4}\right)-\zeta\left(1+q,\frac{3}{4}\right)\right)\\
=&\frac{1}{4^{1+q}q!}\Gamma(1+q)\left(\psi^{(q)}\left(\frac{1}{4}\right)
-\psi^{(q)}\left(\frac{3}{4}\right)\right).
\end{align*}
By the  recurrence relations and initial values above, combining with Proposition \ref{MT1} and Proposition \ref{MT2}, then
\begin{align*}
\int_{0}^{1}\frac{x\ln^{q}(x)}{(1+x^{2})^{r}}dx=&
l_{0}\ln2+\sum_{i=1}^{q}l_{i}\zeta(i+1)+l_{q+1},\\
\int_{0}^{1}\frac{\ln^{2k-1}(x)}{(1+x^{2})^{r}}dx=&
l_{0}G+\sum_{i=1}^{k}l_{i}\pi^{2i-1}
+\sum_{j=2}^{k}c_{j}
\left(\psi^{(2j-1)}\left(\frac{1}{4}\right)
-\psi^{(2j-1)}\left(\frac{3}{4}\right)\right)+l_{k+1},\\
\int_{0}^{1}\frac{\ln^{2k}(x)}{(1+x^{2})^{r}}dx=&
l_{0}G+\sum_{i=1}^{k+1}l_{i}\pi^{2i-1}
+\sum_{j=2}^{k}c_{j}
\left(\psi^{(2j-1)}\left(\frac{1}{4}\right)
-\psi^{(2j-1)}\left(\frac{3}{4}\right)\right)+l_{k+2}.
\end{align*}
\\For the general parameter $a$, by partial fraction theory, \vspace{-0.1cm}
\begin{align}
\int_{0}^{1}\frac{x^{2n-1}\ln^{q}(x)}{(1+x^{2})^{r}}dx=&\int_{0}^{1}x^{m}\ln^{q}(x)dx
+\sum_{i=1}^{r}C_{i}\int_{0}^{1}\frac{x\ln^{q}(x)}{(1+x^{2})^{i}}dx,\label{YBA1}\\
\int_{0}^{1}\frac{x^{2n}\ln^{q}(x)}{(1+x^{2})^{r}}dx=&\int_{0}^{1}x^{m}\ln^{q}(x)dx
+\sum_{i=1}^{r}D_{i}\int_{0}^{1}\frac{\ln^{q}(x)}{(1+x^{2})^{i}}dx\label{YBA2},
\end{align}
where the first integral is
\[\int_{0}^{1}x^{m}\ln^{q}(x)dx=(-1)^{q}(1+m)^{-1-q}\Gamma(1+q),\]
which is a rational number,  hence Theorem
\ref{DTTH1} is proved.
\end{proof}

\begin{theorem}\label{MT4}
Let $a-p<2r,$ then
\begin{align}
\int_{0}^{1}\frac{ x^{a-p}(\arctan x)^{p}}{(1+x^2)^{r}}dx=c_{0}+c_{1}\pi+c_{2}\pi^{2}+\ldots+c_{p}\pi^{p}+c_{p+1}\pi^{p+1},
\end{align}
where  rational coefficients $c_{i}(i=1,2,3,\ldots)$ depend on the values of $p, r$.
\end{theorem}
\begin{proof}
 By partial fraction theory, then
\begin{align*}
\int_{0}^{1}\frac{x^{a-p}(\arctan x)^{p}}{(1+x^{2})^{r}}dx=\sum_{i=1}^{r}C_{i}\int_{0}^{1}\frac{\arctan^{p}(x)}{(1+x^{2})^{i}}dx
+\sum_{i=1}^{r}D_{i}\int_{0}^{1}\frac{x\arctan^{p}(x)}{(1+x^{2})^{i}} ,
\end{align*}
 where $C_{i}, D_{i}(i=1,2,3,\cdot\cdot\cdot)$ depend on $a, r$ and belong to $\mathbb{Q} $.
Since
\[\int_{0}^{1}\frac{x(\arctan(x))^{p}}{(1+x^{2})^{r}}dx=
\frac{p}{2(r-1)}\int_{0}^{1}\frac{(\arctan x)^{p-1}}{(1+x^{2})^{r}}dx-\frac{\pi^{p}}{2^{r}(r-1)4^{p}},\]
 we only need to  deal with the second integral above.  Let $x = \tan \theta$, then
\begin{align}\label{ZMD}
\int_{0}^{1}\frac{ (\arctan x)^{p}}{(1+x^2)^{r}}dx=
&\frac{2r-3}{2r-2}\int_{0}^{1}\frac{ (\arctan x)^{p}}{(1+x^2)^{r-1}}dx-\frac{p(p-1)}{(2r-2)^{2}}\int_{0}^{1}\frac{ (\arctan x)^{p-2}}{(1+x^2)^{r}}dx\nonumber\\
&+\frac{1}{2^{r}(r-1)}\left[\left(\frac{\pi}{4}\right)^{p}
+\frac{p\pi^{p-1}}{4^{p-1}(2r-2)}\right].
\end{align}
According to \eqref{ZMD},  the integral depends on
\[\int_{0}^{1}\frac{ 1}{(1+x^2)^{r}}dx,\int_{0}^{1}\frac{ (\arctan x)}{(1+x^2)^{r}}dx, \int_{0}^{1}\frac{ (\arctan x)^{p}}{(1+x^2)}dx.\]
Since\\
$(1)$
\[\int_{0}^{1}\frac{ 1}{(1+x^2)^{r}}dx=\int_{0}^{\frac{\pi}{4}}\cos^{2r-2}\theta d \theta=c_{0}+c_{1}\pi.\]
\\$(2)$
\begin{align*}
\int_{0}^{1}\frac{ (\arctan x)}{(1+x^2)^{r}}dx=&\int_{0}^{\frac{\pi}{4}}\theta\cos^{2r-2}\theta d \theta
=\int_{0}^{\frac{\pi}{4}}\theta\cos^{2r-4}\theta d \theta\\
&+\frac{\pi}{(2r-3)2^{r+1}}+\frac{1}{(2r-3)(r-1)2^{r}}-\frac{1}{(2r-3)(2r-2)},
\end{align*}
from the recurrence relation above and initial value $\int_{0}^{\frac{\pi}{4}}\theta\cos^{2}\theta d \theta=\frac{1}{64}(-8+4\pi+\pi^{2}),$
one gets
\[\int_{0}^{1}\frac{ (\arctan x)}{(1+x^2)^{r}}dx=c_{0}+c_{1}\pi+c_{2}\pi^{2}.\]
$(3)$
\[\int_{0}^{1}\frac{ (\arctan x)^{p}}{(1+x^2)}dx=\int_{0}^{\frac{\pi}{4}}\theta^{p}d\theta
=\frac{\pi^{p+1}}{4^{p+1}(p+1)}.\]
Then \eqref{ZMD} can be expressed as
\vspace{-0.1cm}
\[\int_{0}^{1}\frac{ (\arctan x)^{p}}{(1+x^2)^{r}}dx=c_{0}+c_{1}\pi+c_{2}\pi^{2}+\ldots+c_{p}\pi^{p}+c_{p+1}\pi^{p+1},\]
and\vspace{-0.1cm}
\[\int_{0}^{1}\frac{ x(\arctan x)^{p}}{(1+x^2)^{r}}dx=c_{0}+c_{1}\pi+c_{2}\pi^{2}+\ldots+c_{p}\pi^{p},\]
hence  Theorem \ref{MT4} is proved.
\end{proof}

\vspace{-0.5cm}
\section{The relation between $h_{n}$ and $t_{n}$ based on $I(a,p,q,r)$ }
\begin{theorem}\label{TH1}For $ k\in  \mathbb{N}, $
\begin{align}\label{EQTH1.1}
 \sum_{n=1}^{\infty}\frac{(-1)^{n+1}h_{n}}{n(2n+2k+1)}
 =&\sum_{n=0}^{\infty}\frac{(-1)^{n}t_{n}(1)}{2(n+1)(n+2k-1)}
 =\frac{1}{2(2k+1)}
	\left[ \frac{\pi^{2}}{8}-\sum_{j=1}^{k}(-1)^{j+1}\right.\nonumber\\
&\left.\times\frac{\pi+H_{\frac{2k-2j-1}{4}}-H_{\frac{2k-2j+1}{4}}}{2k-2j+2}
-(-1)^{k}2G+\frac{(-1)^{k}\pi\ln2}{2}\right],
\\
\label{EQTH1.2}
\sum_{n=1}^{\infty}\frac{(-1)^{n+1}h_{n}}{n(2n+2k+2)}
=&\sum_{n=0}^{\infty}\frac{(-1)^{n}t_{n}(1)}{2(n+1)(n+2k)}
	=\frac{1}{4(k+1)}\left[ \frac{\pi^{2}}{8}-\sum_{j=1}^{k+1}(-1)^{j+1} \right.\nonumber\\ &\left.\times\frac{\pi+H_{\frac{k-j}{2}}-H_{\frac{k-j+1}{2}}}{2k-2j+3}
-\frac{\pi^{2}(-1)^{k+1}}{8}\right]
 .\end{align}
\end{theorem}
\begin{proof}
 By the Taylor expansion in \cite{SN19}\vspace{-0.3cm}
\[(\arctan (x))^{2}=\sum_{n=1}^{\infty}\frac{(-1)^{n+1}h_{n}x^{2n}}{n},\]
then\vspace{-0.3cm}
\begin{align*}\int_{0}^{1}x^{a}(\arctan (x))^{2}dx=&\int_{0}^{1}\sum_{n=1}^{\infty}\frac{(-1)^{n+1}h_{n}x^{2n+a}}{n}dx\\
=&\sum_{n=1}^{\infty}\frac{(-1)^{n+1}h_{n}}{n}\int_{0}^{1}x^{2n+a}dx
=\sum_{n=1}^{\infty}\frac{(-1)^{n+1}h_{n}}{n(2n+a+1)}.
\end{align*}
On the other hand, by Theorem \ref{CX1.1},  \vspace{-0.2cm}
\[\left(\frac{\arctan(x)}{x}\right)^{2}
=\frac{1}{2}\sum_{n=0}^{\infty}\frac{(-x^{2})^{n}t_{n}(1)}{n+1}.\]
Thus the integral  can be rewritten as
\begin{align*}\int_{0}^{1}x^{a}(\arctan (x))^{2}dx=&\int_{0}^{1}x^{a-2}\left(\frac{\arctan(x)}{x}\right)^{2}dx\\
=&
\frac{1}{2}\sum_{n=0}^{\infty}\frac{(-1)^{n}t_{n}(1)}{n+1}\int_{0}^{1}x^{2n+a-2}dx
\\
=&\frac{1}{2}\sum_{n=0}^{\infty}\frac{(-1)^{n}t_{n}(1)}{(n+1)(2n+a-1),}
\end{align*}
by \eqref{EQPTH1.1} and \eqref{EQPTH1.2}, Theorem \ref{TH1} is proved.
\end{proof}

\begin{example} Let $k=1,2$ in  \eqref{EQTH1.1} and \eqref{EQTH1.2}, then
\begin{align*}
\sum_{n=1}^{\infty}\frac{(-1)^{n+1}h_{n}}{n(2n+3)}
=&\sum_{n=0}^{\infty}\frac{(-1)^{n}t_{n}(1)}{2(n+1)(2n+1)}
=\frac{1}{48}(16+16G+\pi^{2}-4\pi(2+\ln 2)),\\
\sum_{n=1}^{\infty}\frac{(-1)^{n+1}h_{n}}{n(2n+4)}=&
\sum_{n=0}^{\infty}\frac{(-1)^{n}t_{n}(1)}{2(n+1)(2n+2)}
=\frac{1}{12}(1+\pi-4\ln 2),
\\
\sum_{n=1}^{\infty}\frac{(-1)^{n+1}h_{n}}{n(2n+5)}=&
\sum_{n=0}^{\infty}\frac{(-1)^{n}t_{n}(1)}{2(n+1)(2n+3)}
=\frac{1}{240}(-64-48G+3\pi(8+\pi+\ln 16)),\\
\sum_{n=1}^{\infty}\frac{(-1)^{n+1}h_{n}}{n(2n+6)}=&
\sum_{n=0}^{\infty}\frac{(-1)^{n}t_{n}(1)}{2(n+1)(2n+4)}
=\frac{1}{720}(-52-52\pi+15\pi^{2}+184\ln 2).
\end{align*}
\end{example}

\begin{theorem}\label{TH2}For $k\in\mathbb{N}$, then
\begin{align}\label{EQTH2.1}
\sum_{j=1}^{\infty}\sum_{i=1}^{j}\frac{(-1)^{j+1}3h_{i}}{i(2j+1)(2j+2k+2)}
=&\frac{3}{4}\sum_{n=0}^{\infty}\frac{(-1)^{n}t_{n}(2)}{(n+\frac{3}{2})(2n+2k-2)}=
\frac{\pi^{3}}{64(2k+1)}
\nonumber\\&-\frac{3}{2k+1}\sum_{j=1}^{k}\frac{(-1)^{j+1}}{4(k+1-j)}
\left[ \frac{\pi^{2}}{8}-(-1)^{k+1-j}\frac{\pi^{2}}{8}\right.\nonumber
\end{align}
\begin{align}
-\sum_{j'=1}^{k+1-j}(-1)^{j'+1} \left. \frac{\pi+H_{\frac{k-j-j'}{2}}-H_{\frac{k-j-j'+1}{2}}}{2k-2j-2j'+3}\right]
-\frac{(-1)^{k}3}{64(2k+1)}(16G\pi-\pi^{2}\ln4-21\zeta(3)),
\end{align}\vspace{-0.8cm}
\begin{align}
\label{EQTH2.2}
\sum_{j=1}^{\infty}\sum_{i=1}^{j}\frac{(-1)^{j+1}3h_{i}}{i(2j+1)(2j+2k+3)}
=&\frac{3}{4}\sum_{n=0}^{\infty}\frac{(-1)^{n}t_{n}(2)}{(n+\frac{3}{2})(2n+2k-1)}
=\frac{\pi^{3}}{128(k+1)}
\nonumber
\\
&-\frac{3}{2k+2}\sum_{j=1}^{k+1}\frac{(-1)^{j+1}}{2(2k+3-2j)}
\left[ \frac{\pi^{2}}{8}-(-1)^{k+1-j}\right.\nonumber\\
&\times2G+\frac{(-1)^{k+1-j}\pi\ln2}{2}-\sum_{j'=1}^{k+1-j}(-1)^{j'+1}
  \nonumber
\\
&\left. \times\frac{\pi+H_{\frac{2k-2j-2j'+1}{4}}-H_{\frac{2k-2j-2j'+3}{4}}}{2k-2j-2j'+4}\right]
 -\frac{(-1)^{k+1}\pi^{3}}{128(k+1)}.
\end{align}\end{theorem}
\begin{proof}Using the Taylor expansion in \cite{SN20}\vspace{-0.3cm}
\[(\arctan (x))^{3}=3\sum_{j=1}^{\infty}\frac{(-1)^{j+1}}{2j+1}
\sum_{i=1}^{j}\frac{h_{i}}{i}x^{2j+1},\]\vspace{-0.3cm}
then\vspace{-0.1cm}
\begin{align*}
\int_{0}^{1}x^a (\arctan (x))^{3}dx=&\int_{0}^{1}3\sum_{j=1}^{\infty}\frac{(-1)^{j+1}}{2j+1}
\sum_{i=1}^{j}\frac{h_{i}}{i}x^{2j+1+a}dx
\\=&\sum_{j=1}^{\infty}\sum_{i=1}^{j}
\frac{(-1)^{j+1}3h_{i}}{i(2j+1)(2j+a+2)}.
\end{align*}
We can also calculate the integral as follows,
\begin{align*}
\int_{0}^{1}x^{a-3}\left(\frac{\arctan (x)}{x}\right)^{3}dx
=&\int_{0}^{1}\frac{3}{4}\sum_{n=0}^{\infty}\frac{(-1)^{n}t_{n}(2)}{n+\frac{3}{2}}x^{2n+a-3}dx\\
=&\frac{3}{4}\sum_{n=0}^{\infty}\frac{(-1)^{n}t_{n}(2)}{(n+\frac{3}{2})(2n+a-2)},
\end{align*}
by \eqref{EQPTH2.1} and \eqref{EQPTH2.2}, Theorem \ref{TH2} is proved.
\end{proof}

\begin{example}Let $k=1,2$ in \eqref{EQTH2.1} and \eqref{EQTH2.2}, then
\begin{align*}
\sum_{j=1}^{\infty}\sum_{i=1}^{j}
\frac{(-1)^{j+1}3h_{i}}{i(2j+1)(2j+4)}=&
\frac{3}{4}\sum_{n=0}^{\infty}\frac{(-1)^{n}t_{n}(2)}{(n+\frac{3}{2})(2n)}=
\frac{1}{192}(48(1+G)\pi+\pi^{3}
\\&-96\ln 2-6\pi^{2}(2+\ln 2)-63\zeta (3)),\\
\sum_{j=1}^{\infty}\sum_{i=1}^{j}
\frac{(-1)^{j+1}3h_{i}}{i(2j+1)(2j+5)}=&
\frac{3}{4}\sum_{n=0}^{\infty}\frac{(-1)^{n}t_{n}(2)}{(n+\frac{3}{2})(2n+1)}\\
=&
\frac{1}{32}(-8-32G+\pi(4+\pi+\ln(256))),
\\
\sum_{j=1}^{\infty}\sum_{i=1}^{j}
\frac{(-1)^{j+1}3h_{i}}{i(2j+1)(2j+6)}=&
\frac{3}{4}\sum_{n=0}^{\infty}\frac{(-1)^{n}t_{n}(2)}{(n+\frac{3}{2})(2n+2)}
=\frac{1}{320}(-16+160\ln 2+\pi(-64\\
&-48G+\pi(12+\pi+\ln(64)))+63\zeta(3)),\\
\sum_{j=1}^{\infty}\sum_{i=1}^{j}
\frac{(-1)^{j+1}3h_{i}}{i(2j+1)(2j+7)}=&
\frac{3}{4}\sum_{n=0}^{\infty}\frac{(-1)^{n}t_{n}(2)}{(n+\frac{3}{2})(2n+3)}
\\
=&
\frac{1}{960}(288+736G-26\pi^{2}+5\pi^{3}-8\pi(16+23\ln2)).
\end{align*}
\end{example}

\begin{theorem}\label{TH4}
\begin{align}\label{EQTH4.1}
\sum_{n=1}^{\infty}\frac{(-1)^{n}h_{n}}{n(2n+2k+1)^{2}}=&
\sum_{n=0}^{\infty}\frac{(-1)^{n+1}t_{n}(1)}{2(n+1)(2n+2k-1)^{2}}
=\frac{1}{2k+1}\left\{-\frac{\pi^2}{16}  + \frac{k}{2k+1}\right.
\nonumber
\\
&\times\left( \frac{\pi^2}{8}-(-1)^{k}2G+\frac{(-1)^{k}\pi \ln2}{2}  -\sum_{j=1}^{k}(-1)^{j+1}\right.\nonumber
\\
&\times\left.\frac{\pi+H_{\frac{2k-2j-1}{4}}-H_{\frac{2k-2j+1}{4}}}{2+2k-2j} \right)-2\left[\sum_{j=1}^{k}\frac{(-1)^{j+1}}{64(k+1-j)^{2}}\nonumber\right.
\\
&\times\left( -4\pi
+4\psi^{(0)}\left(\frac{2k+5-2j}{4}\right)\right.
-4\psi^{(0)}\left(\frac{2k+3-2j}{4}\right)\nonumber\\
&-2(k+1-j)\left(\psi^{(1)}\left(\frac{2k+5-2j}{4}\right)
\left.-\psi^{(1)}\left(\frac{2k+3-2j}{4}\right)\right)\right)
\nonumber\\
&+\left.\frac{(-1)^{k}}{96}\left(3\pi^{3} +6\pi(\ln2)^2 -192W(3)\right)\right]+2\left[ \sum_{j=1}^{k}(-1)^{j+1}\right.
\nonumber\\
&\times\frac{\pi+H_{\frac{2k-1-2j}{4}}
-H_{\frac{2k+1-2j}{4}}}{8(k+1-j)}\left.\left.
+\frac{(-1)^{k}}{8}(4G-\pi\ln2)\right]\right\},
\\
\label{EQTH4.2}
\sum_{n=1}^{\infty}\frac{(-1)^{n}h_{n}}{n(2n+2k+2)^{2}}	=&
\sum_{n=0}^{\infty}\frac{(-1)^{n+1}t_{n}(1)}{2(n+1)(2n+2k)^{2}}
=\frac{1}{2(k+1)}\left\{-\frac{\pi^2}{16}+ \frac{2k+1}{4(k+1)}\right.\nonumber\\
&\times
\left( \frac{\pi^2}{8}-\sum_{j=1}^{k+1}(-1)^{j+1}
\frac{\pi+H_{\frac{k-j}{2}}-H_{\frac{k-j+1}{2}}}{2k-2j+3}\right. \nonumber\\ &-\left.\frac{(-1)^{k+1}\pi^{2}}{8}\right)
-2\left[\sum_{j=1}^{k+1}\right.\frac{(-1)^{j+1}}{16(2k+3-2j)^{2}} \left( -4\pi\right.\nonumber
\end{align}
\begin{align}
 &+4\psi^{(0)}\left(\frac{k+3-j}{2}\right)-4\psi^{(0)}\left(\frac{k+2-j}{2}\right)
-(2k+3-2j)\left(\psi^{(1)}\left(\frac{k+3-j}{2}\right)\right.\nonumber\\
&-\left.\left.\left.\psi^{(1)}\left(\frac{k+2-j}{2}\right)\right)\right)+\frac{(-1)^{k+1}}{16}(-4G\pi+7\zeta(3))\right]
\nonumber\\
&+2\left[ \sum_{j=1}^{k+1}(-1)^{j+1}\frac{\pi+H_{\frac{k-j}{2}}-H_{\frac{k+1-j}{2}}}{4(2k+3-2j)}\right.
\left.\left.+\frac{(-1)^{k+1}\pi^2}{32}\right] \right\}.
\end{align}
\end{theorem}
\begin{proof}
 Using the Taylor Expansion of $(\arctan (x))^{2}$ and integrating by parts, then
\begin{align*}
\int_{0}^{1}x^{a}(\arctan (x))^{2}\ln (x)dx=&\int_{0}^{1}\sum_{n=1}^{\infty}\frac{(-1)^{n+1}h_{n}}{n}x^{2n+a}\ln (x)dx\\
=&\sum_{n=1}^{\infty}\frac{(-1)^{n}h_{n}}{n(2n+a+1)^{2}},
\end{align*}
at the same time,\vspace{-0.3cm}
\begin{align*}
&\int_{0}^{1}x^{a}(\arctan (x))^{2}\ln (x)dx=\int_{0}^{1}x^{a-2}\left(\frac{\arctan(x)}{x}\right)^{2}\ln(x)dx
\\
&=\int_{0}^{1}\frac{1}{2}\sum_{n=0}^{\infty}\frac{(-1)^{n}t_{n}(1)}{n+1}
x^{2n+a-2}\ln(x)dx
=\frac{1}{2}\sum_{n=0}^{\infty}\frac{(-1)^{n+1}t_{n}(1)}{(n+1)(2n+a-1)^{2}},
\end{align*}
by \eqref{EQPTH4.1} and \eqref{EQPTH4.2}, and Theorem \ref{TH4} is proved.
\end{proof}
\begin{example}Let $k=1,2$ in \eqref{EQTH4.1} and \eqref{EQTH4.2},  then
\begin{align*}
\sum_{n=1}^{\infty}\frac{(-1)^{n}h_{n}}{n(2n+3)^{2}}
=&\sum_{n=0}^{\infty}\frac{(-1)^{n+1}t_{n}(1)}{2(n+1)(2n+1)^{2}}
=\frac{1}{144}(-88+32G+3\pi^{3}-\pi^{2}
\\
&+\pi(20+6(\ln 2)^{2}+\ln 16)-192W(3)),\\
\sum_{n=1}^{\infty}\frac{(-1)^{n}h_{n}}{n(2n+4)^{2}}=&
\sum_{n=0}^{\infty}\frac{(-1)^{n+1}t_{n}(1)}{2(n+1)(2n+2)^{2}}\\
=&\frac{1}{288}(-26+2\pi(-19+18G+2\pi)+104\ln 2-63\zeta (3)),
\\
\sum_{n=1}^{\infty}\frac{(-1)^{n}h_{n}}{n(2n+5)^{2}}=&
\sum_{n=0}^{\infty}\frac{(-1)^{n+1}t_{n}(1)}{2(n+1)(2n+3)^{2}}
=\frac{1}{3600}\left(1652-936G-45\pi^{3}-9\pi^{2} \right.
\\&\left.-18\pi(14+5(\ln 2)^{2}+\ln 4)+2880W(3)\right),
\\\sum_{n=1}^{\infty}\frac{(-1)^{n}h_{n}}{n(2n+6)^{2}}=&
\sum_{n=0}^{\infty}\frac{(-1)^{n+1}t_{n}(1)}{2(n+1)(2n+4)^{2}}\\
=&\frac{1602+(1932-1800G-305\pi)\pi-5064\ln 2+3150\zeta(3)}{21600}.
\end{align*}
\end{example}

\begin{theorem}\label{TH5}
\begin{align}\label{EQTH5}
\sum_{n=1}^{\infty}\sum_{j=0}^{\infty}\frac{(-1)^{n+j}h_{n}}{n(2j+2n+2k+1)^{2}}
=&\sum_{n=0}^{\infty}\sum_{j=0}^{\infty}\frac{(-1)^{n+j+1}t_{n}(1)}{2(n+1)(2n+2j+2k-1)^{2}}
\nonumber\\
=&\sum_{j=1}^{k}\frac{(-1)^{j+1}}{2k-2j+1}\times\left\{ -\frac{\pi^{2}}{16}+\frac{k-j}{2k-2j+1}\right.\nonumber\\
&
\times\left(\frac{\pi^{2}}{8}-(-1)^{k-j}2G+\frac{(-1)^{k-j}}{2}\pi\ln2\nonumber\right.
\\
&-\sum_{j'=1}^{k-j}(-1)^{j'+1}
\left.\frac{\pi+H_{\frac{2k-2j-2j'-1}{4}}-H_{\frac{2k-2j-2j'+1}{4}}}{2k+2-2j-2j'} \right) \nonumber\\
&-2\left[\sum_{j'=1}^{k-j}\frac{(-1)^{j'+1}}{64(k+1-j-j')^{2}}\times\left(-4\pi\right.\right.
\nonumber\\
&+4\psi^{(0)}\left(\frac{2k+5-2j-2j'}{4}\right)\nonumber\\
&-4\psi^{(0)}\left(\frac{2k+3-2j-2j'}{4}\right)\nonumber\\
&-2(k+1-j-j')\left(\psi^{(1)}\left(\frac{2k+5-2j-2j'}{4}\right)\right.
\nonumber
\\
&-\left.\left.\psi^{(1)}\left(\frac{2k+3-2j-2j'}{4}\right)\right)\right)
+\frac{(-1)^{k-j}}{96}\nonumber\\
&\left.\times(3\pi^{3}+6\pi(\ln2)^2
-192W(3))\right]\nonumber\\
&+2
\left[\sum_{j'=1}^{k-j}(-1)^{j'+1}\frac{\pi+H_{\frac{2k-1-2j-2j'}{4}}
-H_{\frac{2k+1-2j-2j'}{4}}}{8(k+1-j-j')}\right.\nonumber\\
&+\left.\left.(-1)^{k-j}\frac{4G-\pi\ln2}{8}\right]\right\}
+\frac{(-1)^{k}}{48}(-3G\pi^{2}+24\beta(4)).
\end{align}
\end{theorem}
\begin{proof}Using the Taylor expansion of $\arctan(x)$ and $\frac{\arctan(x)}{x},$ then
\begin{align*}
I(a,2,1,1)=&\int_{0}^{1}\frac{x^{a}\ln (x)(\arctan (x))^{2}}{1+x^{2}}dx=\int_{0}^{1}\sum_{n=1}^{\infty}\sum_{j=0}^{\infty}(-1)^{j}\frac{(-1)^{n+1}h_{n}}{n}\\
&\times x^{2j+2n+a}\ln (x)dx
=\sum_{n=1}^{\infty}\sum_{j=0}^{\infty}\frac{(-1)^{n+j}h_{n}}{n(2j+2n+a+1)^{2}},
\end{align*}
and
\begin{align*}
\int_{0}^{1}\frac{x^{a-2}\ln (x)\left(\frac{\arctan (x)}{x}\right)^{2}}{1+x^{2}}dx=&\int_{0}^{1}\frac{1}{2}
\sum_{n=0}^{\infty}\frac{(-1)^{n}t_{n}(1)}{n+1}\sum_{j=0}^{\infty}(-1)^{j}x^{2n+2j+a-2}\ln(x)dx
\\=&\frac{1}{2}
\sum_{n=0}^{\infty}\sum_{j=0}^{\infty}\frac{(-1)^{n+j+1}t_{n}(1)}{(n+1)(2n+2j+a-1)^{2}},
\end{align*}
by \eqref{EQPTH5}, Theorem \ref{TH5} is proved.
\end{proof}
\begin{example}Let $k=1,2,3$ in  \eqref{EQTH5}, then
\begin{align*}
\sum_{n=1}^{\infty}\sum_{j=0}^{\infty}\frac{(-1)^{n+j}h_{n}}{n(2j+2n+3)^{2}}=&
\sum_{n=0}^{\infty}\sum_{j=0}^{\infty}\frac{(-1)^{n+j+1}t_{n}(1)}{2(n+1)(2n+2j+1)^{2}}
\\=&
\frac{1}{3072}\left(192G(16+\pi^{2})-64(3\pi^{3}
+3\pi^{2}\right.
\\&+6\pi\ln 2(2+\ln 2))-1536\beta(4)\left.+1228W(3)\right),
\\
\sum_{n=1}^{\infty}\sum_{j=0}^{\infty}\frac{(-1)^{n+j}h_{n}}{n(2j+2n+5)^{2}}
=&\sum_{n=0}^{\infty}\sum_{j=0}^{\infty}\frac{(-1)^{n+j+1}t_{n}(1)}{2(n+1)(2n+2j+2)^{2}}
\\
=&\frac{1}{9216}\left(-64G(112+9\pi^{2})+4068\beta(4)\right.-5632
\\
&+256(3\pi^{3}+2\pi^{2}\left.+\pi(5+\ln 4(5+\ln 8)))-49152W(3)\right),
\\
\sum_{n=1}^{\infty}\sum_{j=0}^{\infty}\frac{(-1)^{n+j}h_{n}}{n(2j+2n+7)^{2}}
=&\sum_{n=0}^{\infty}\sum_{j=0}^{\infty}\frac{(-1)^{n+j+1}t_{n}(1)}{2(n+1)(2n+2j+3)^{2}}
=\frac{1}{230400}\left(-115200\beta(4)\right.
\\
&-64(\pi(752+\pi(209+345\pi))+1036\pi\ln 2+690\pi(\ln 2)^{2})
\\&+192\left(1284+7360W(3)\right)\left.+64G(1864+225\pi^{2})\right).
\end{align*}
\end{example}

\begin{theorem}\label{TH6}
\begin{align}\label{EQTH6}
\sum_{j=1}^{\infty}\sum_{i=1}^{j}\frac{(-1)^{j}3h_{i}}{i(2j+1)(2j+2k+3)^{2}}
=&\frac{3}{4}\sum_{n=0}^{\infty}\frac{(-1)^{n+1}t_{n}(2)}{(n+\frac{3}{2})(2n+2k-1)^{2}}
=-\frac{\pi^{3}}{128(k+1)}
\nonumber\\
&+\frac{(2k+1)\pi^{3}}{256(k+1)^{2}}
-\frac{3(2k+1)}{4(k+1)^{2}}\sum_{j=1}^{k+1}\frac{(-1)^{j+1}}{2(2k+3-2j)}\nonumber\\
&\times\left(  \frac{\pi^{2}}{8}-\sum_{j'=1}^{k+1-j}
\frac{\pi+H_{\frac{2k-2j-2j'+1}{4}}-H_{\frac{2k-2j-2j'+3}{4}}}{2(k-j-j'+2)}\right.\nonumber
\\
&\times(-1)^{j'+1}-(-1)^{k+1-j}2G\left.+\frac{(-1)^{k+1-j}\pi\ln2}{2}\right)
\nonumber
\end{align}
\begin{align}
&-\frac{(-1)^{k+1}(2k+1)\pi^{3}}{256(k+1)^{2}}
-\frac{3}{2(k+1)}\sum_{j=1}^{k+1}\frac{(-1)^{j+1}}{2k+3-2j}\times\left\{-\frac{\pi^2}{16}  + \frac{k+1-j}{2k+3-2j}\right.\nonumber\\
&\times\left( \frac{\pi^2}{8}\right.-\sum_{j'=1}^{k+1-j}(-1)^{j'+1}
\frac{\pi+H_{\frac{2k-2j-2j'+1}{4}}-H_{\frac{2k-2j-2j'+3}{4}}}{2(k-j-j'+2)} -(-1)^{k+1-j}2G\nonumber\\
& \left. +\frac{(-1)^{k+1-j}\pi \ln2}{2}\right)-2\left[\sum_{j'=1}^{k+1-j}
\frac{(-1)^{j'+1}}{64(k+2-j-j')^{2}}\times \left(-4\pi \right.\right.\nonumber
\\
&+4\psi^{(0)}\left(\frac{2k+7-2j-2j'}{4}\right)-4\psi^{(0)}\left(\frac{2k+5-2j-2j'}{4}\right)-2(k+2-j-j')\nonumber\\
&\times\psi^{(1)}\left(\frac{2k+7-2j-2j'}{4}\right)+2(k+2-j-j')
\left.\times\psi^{(1)}\left(\frac{2k+5-2j-2j'}{4}\right)\right)\nonumber\\
&+\frac{(-1)^{k+1-j}}{96}(3\pi^{3}\left.+6\pi(\ln2)^2-192W(3))\right]+2\left[ \sum_{j'=1}^{k+1-j}(-1)^{j'+1}\right.
\nonumber
\\
&\times\frac{\pi+H_{\frac{2k+1-2j-2j'}{4}}-H_{\frac{2k+3-2j-2j'}{4}}}{8(k+2-j-j')}
+(-1)^{k+1-j}\left(\frac{G}{2}\right.-\left.\left.\left.\frac{\pi\ln2}{8}\right)\right]\right\}
\nonumber\\
&-\frac{(-1)^{k+1}(-3G\pi^{2}+24\beta(4))}{32(k+1)}+\frac{3}{2(k+1)}\sum_{j=1}^{k+1}\frac{(-1)^{j+1}}{2(2k+3-2j)}\left(\frac{\pi^{2}}{8}
\right.
\nonumber
\\
&-(-1)^{k+1-j}2G+\frac{(-1)^{k+1-j}}{2}\pi\ln2-\sum_{j'=1}^{k+1-j}(-1)^{j'+1}\nonumber
\\
&\left.\times\frac{\pi+H_{\frac{2k+1-2j-2j'}{4}}-H_{\frac{2k+3-2j-2j'}{4}}}{2(k+2-j-j')}
\right)
+\frac{(-1)^{k+1}\pi^{3}}{128(k+1)}.
\end{align}
\end{theorem}
\begin{proof}
Using the Taylor expansion of $(\arctan (x))^{3}$ and integrating by parts, then
\begin{align*}
\int_{0}^{1}x^{a}\ln (x)(\arctan(x))^3dx=&\int_{0}^{1}3\sum_{j=1}^{\infty}\frac{(-1)^{j+1}}{j+1}
\sum_{i=1}^{j}\frac{h_{i}}{i}x^{2j+1+a}\ln(x)dx\\
=&\sum_{j=1}^{\infty}\frac{(-1)^{j+1}}{2j+1}
\sum_{i=1}^{j}\frac{-3h_{i}}{i(2j+a+2)^{2}},
\end{align*}
we can also use the Taylor expansion of $\left(\frac{\arctan (x)}{x}\right)^{3}$ and integrating by parts, then
\begin{align*}
\int_{0}^{1}x^{a}\ln (x)(\arctan(x))^3dx=&\int_{0}^{1}x^{a-3}\ln (x)\left(\frac{\arctan(x)}{x^{3}}\right)^3dx\end{align*}
\begin{align*}
=\int_{0}^{1}\frac{3}{4}\sum_{n=0}^{\infty}\frac{(-1)^{n}t_{n}(2)}{n+\frac{3}{2}}x^{2n+a-3}\ln(x)dx
=\frac{3}{4}\sum_{n=0}^{\infty}\frac{(-1)^{n+1}t_{n}(2)}{(n+\frac{3}{2})(2n+a-2)^{2}},
\end{align*}
by \eqref{EQPTH6}, and Theorem \ref{TH6} is proved.
\end{proof}
\begin{example}Let $k=0,1,2$ in \eqref{EQTH6}, then
\begin{align*}
\sum_{j=1}^{\infty}\sum_{i=1}^{j}\frac{(-1)^{j}3h_{i}}{i(2j+1)(2j+3)^{2}}
=&\frac{3}{4}\sum_{n=0}^{\infty}\frac{(-1)^{n+1}t_{n}(2)}{(n+\frac{3}{2})(2n-1)^{2}}
\\=&\frac{1}{2048}\left(176\pi^{3}+1536\beta(4)+288\pi^{2}\right.
\\
&+384\pi\ln 2(3+\ln 2)\left.-192G(24+\pi^{2})-12288W(3)\right),
\\
\sum_{j=1}^{\infty}\sum_{i=1}^{j}\frac{(-1)^{j}3h_{i}}{i(2j+1)(2j+5)^{2}}
=&\frac{3}{4}\sum_{n=0}^{\infty}\frac{(-1)^{n+1}t_{n}(2)}{(n+\frac{3}{2})(2n+1)^{2}}
=\frac{1}{12288}(64G(160+9\pi^{2})\\
&-32(24\pi^{3}+19\pi^{2}+4\pi(13+\ln 4(13+\ln 64)))\\
&-4608\beta(4)+256(25+192W(3))),\\
\sum_{j=1}^{\infty}\sum_{i=1}^{j}\frac{(-1)^{j}3h_{i}}{i(2j+1)(2j+7)^{2}}
=&\frac{3}{4}\sum_{n=0}^{\infty}\frac{(-1)^{n+1}t_{n}(2)}{(n+\frac{3}{2})(2n+3)^{2}}
=\frac{1}{460800}(21680\pi^{3}
\\
&-192G(928+75\pi^{2})+15456\pi^{2}\\
&+384\pi(152+\ln 2(211+115\ln 2))
\\
&\left.+115200\beta(4)-192(1404+7360W(3))\right).
\end{align*}
\end{example}
\vspace{-0.3cm}
\begin{theorem}\label{TH7}
\begin{align}\label{EQTH7.1}
\sum_{n=1}^{\infty}\frac{(-1)^{n+1}2h_{n}}{(2n+2k)^{3}}=&
\sum_{n=0}^{\infty}\sum_{j=0}^{\infty}\frac{(-1)^{n+j}t_{n}(0)}{(n+\frac{1}{2})(2n+2j+2k)^{3}}
=\sum_{j=1}^{k}\frac{(-1)^{j+1}}{64(1+2k-2j)^{3}}\nonumber
\\
&\times\left(32\pi-32\psi^{(0)}\left(\frac{k+2-j}{2}\right)
+32\psi^{(0)}\left(\frac{1+k-j}{2}\right)\right.
\nonumber
\\
&+8(1+2k-2j)\psi^{(1)}\left(\frac{k+2-j}{2}\right)-8(1+2k-2j)\nonumber\\
&\times\psi^{(1)}\left(\frac{1+k-j}{2}\right)-(1+2k-2j)^{2}\nonumber\\
&\times
\left(\psi^{(2)}\left(\frac{k+2-j}{2}\right)\right.
\left.\left.-\psi^{(2)}\left(\frac{1+k-j}{2}\right)\right)\right)
\nonumber\\
&+(-1)^{k}
\left[\frac{1}{2}\left(\frac{11}{4}\zeta(4)
-\frac{7}{4}\zeta(3)\ln2+\frac{1}{2}\zeta(2)\ln^{2}2\right.\right.
\nonumber\\
&
-\left.\left.\frac{1}{12}\ln^{4}2-2Li_{4}\left(\frac{1}{2}\right)\right)-\frac{25}{128}\zeta(4)\right],
\\
\label{EQTH7.2}
\sum_{n=1}^{\infty}\frac{(-1)^{n+1}2h_{n}}{(2n+2k+1)^{3}}=&
\sum_{n=0}^{\infty}\sum_{j=0}^{\infty}\frac{(-1)^{n+j}t_{n}(0)}{(n+\frac{1}{2})(2n+2j+2k+1)^{3}}
=\sum_{j=1}^{k}(-1)^{j+1}\nonumber\\
&\times\left[\frac{1}{512(1+k-j)^{3}}
\left(32\pi-32\psi^{(0)}\left(\frac{2k+5-2j}{4}\right)\right.\right.\nonumber\\
&\left.
+32\psi^{(0)}\left(\frac{3+2k-2j}{4}\right)\right.+16(1+k-j)
\nonumber
\\
&\times\left(\psi^{(1)}\left(\frac{2k+5-2j}{4}\right)
\left.-\psi^{(1)}\left(\frac{3+2k-2j}{4}\right)\right)\right.-4(1+k\nonumber\\
&-j)^{2}\left(\psi^{(2)}\left(\frac{2k+5-2j}{4}\right)\right.
\left.\left.\left.-\psi^{(2)}\left(\frac{3+2k-2j}{4}\right)\right)\right)\right]
\nonumber
\\
&
+(-1)^{k}\left(\frac{7\pi}{64}\zeta(3)+\beta(4)-\frac{\pi^{3}}{16}\ln2\right).
\end{align}
\end{theorem}

\begin{proof}
Using the Taylor expansion in \cite{SF22}\vspace{-0.1cm}
\[\frac{\arctan (x)}{1+x^2}=\sum_{n=1}^{\infty}(-1)^{n+1}\sum_{j=1}^{n}\frac{x^{2n-1}}{2j-1}
=\sum_{n=1}^{\infty}(-1)^{n+1}h_{n}x^{2n-1},\]
and integrating by parts, then\vspace{-0.1cm}
\begin{align*}
\int_{0}^{1}\frac{x^{a}\arctan (x)(\ln x)^{2}}{1+x^2}dx =&\int_{0}^{1}\sum_{n=1}^{\infty}(-1)^{n+1}h_{n}x^{2n-1+a}(\ln x)^{2}dx
=2\sum_{n=1}^{\infty}\frac{(-1)^{n+1}h_{n}}{(2n+a)^{3}},\end{align*}
we can also rewrite the integral as\vspace{-0.1cm}
\begin{align*}
&\int_{0}^{1}\frac{x^{a-1}\left(\frac{\arctan(x)}{x}\right)\ln^{2}(x)}{1+x^{2}}dx
\\
=&\int_{0}^{1}\frac{1}{2}\sum_{n=0}^{\infty}\frac{(-1)^{n}t_{n}(0)}{n+\frac{1}{2}}\sum_{j=0}^{\infty}
(-1)^{j}x^{2n+2j+a-1}\ln^{2}(x)dx\\
=&\sum_{n=0}^{\infty}\sum_{j=0}^{\infty}\frac{(-1)^{n+j}t_{n}(0)}{(n+\frac{1}{2})(2n+2j+a)^{3}},
\end{align*}
by \eqref{EQPTH7.1} and \eqref{EQPTH7.2}, and Theorem \ref{TH7} is proved.
\end{proof}

\begin{example}Let $k=1,2$ in \eqref{EQTH7.1} and \eqref{EQTH7.2}, then
\begin{align*}
\sum_{n=1}^{\infty}\frac{(-1)^{n+1}2h_{n}}{(2n+2)^{3}}=&
\sum_{n=0}^{\infty}\sum_{j=0}^{\infty}\frac{(-1)^{n+j}t_{n}(0)}{(n+\frac{1}{2})(2n+2j+2)^{3}}
=\frac{\pi}{2}
-\frac{151\pi^{4}}{11520}-\frac{1}{24}\pi^{2}(1+(\ln 2)^{2})\\&+Li_{4}\left(\frac{1}{2}\right)-\frac{3}{16}\zeta (3)+\frac{1}{24}\ln2(-24+(\ln 2)^{3}+21\zeta (3)),\\
\sum_{n=1}^{\infty}\frac{(-1)^{n+1}2h_{n}}{(2n+3)^{3}}=&
\sum_{n=0}^{\infty}\sum_{j=0}^{\infty}\frac{(-1)^{n+j}t_{n}(0)}{(n+\frac{1}{2})(2n+2j+3)^{3}}
=\frac{1}{256}\left(\zeta\left(4,\frac{3}{4}\right)-448
\right.\\
&\left.+8\pi^{3}(1+\ln 4)+4\pi(8-7\zeta(3))+128G-\zeta \left(4,\frac{1}{4}\right)\right),
\\
\sum_{n=1}^{\infty}\frac{(-1)^{n+1}2h_{n}}{(2n+4)^{3}}=&
\sum_{n=0}^{\infty}\sum_{j=0}^{\infty}\frac{(-1)^{n+j}t_{n}(0)}{(n+\frac{1}{2})(2n+2j+4)^{3}}
=\frac{1}{34560}(-16640\pi + 453 \pi^{4}\\
&+ 160 \pi^{2} (10 + 9 \ln2^{2}) - 160 (38  + 9 (\ln2)^{4}- 112 \ln4 \\
&+ 216 Li_{4}\left(\frac{1}{2}\right) + 27 (-2 + \ln128) \zeta(3))),\\
\sum_{n=1}^{\infty}\frac{(-1)^{n+1}2h_{n}}{(2n+5)^{3}}=&
\sum_{n=0}^{\infty}\sum_{j=0}^{\infty}\frac{(-1)^{n+j}t_{n}(0)}{(n+\frac{1}{2})(2n+2j+5)^{3}}
\\=&\frac{1}{13824}\left(27\zeta\left(4,\frac{1}{4}\right)-108\pi(8+\pi^{2}(3+\ln 16)-7\zeta(3))\right.\\
&\left.+16336-4320G-27\zeta \left(4,\frac{3}{4}\right)\right).
\end{align*}
\end{example}

\begin{theorem}\label{TH8}
\begin{align}\label{EQTH8.1}
\sum_{n=1}^{\infty}\frac{(-1)^{n+1}2h_{n}}{n(2n+2k+1)^{3}}
=&\sum_{n=0}^{\infty}\frac{(-1)^{n}t_{n}(1)}{(n+1)(2n+2k-1)^{3}}
=-\frac{2}{2k+1}\left\{\sum_{j=1}^{k}(-1)^{j+1}\nonumber\right.\\
&\left.\times\left[\frac{1}{512(k+1-j)^{3}}\left(32\pi
-32\psi^{(0)}\left(\frac{2k+5-2j}{4}\right)\right.\right.\right.\nonumber
\\
&+32\psi^{(0)}\left(\frac{2k+3-2j}{4}\right)
+16(k+1-j)\nonumber\\
&\times\left(\psi^{(1)}\left(\frac{2k+5-2j}{4}\right)-\psi^{(1)}\left(\frac{2k+3-2j}{4}\right)\right)
\nonumber\\&-4(k+1-j)^{2}\left(\psi^{(2)}\left(\frac{2k+5-2j}{4}\right)\right.\nonumber
\\
&-\psi^{(2)}\left.\left.\left.\left(\frac{3+2k-2j}{4}\right)\right)\right)\right]
+(-1)^{k}\left(\frac{7\pi}{64}\zeta(3)\right.\nonumber\\
&+\beta(4)\left.\left.-\frac{\pi^{3}\ln2}{16}\right)\right\}
-\frac{2}{(2k+1)^{2}}\left\{-\frac{\pi^{2}}{16}+\frac{k}{2k+1}\right.\nonumber\\
&\times\left.\left(\frac{\pi^{2}}{8}-\sum_{j=1}^{k}(-1)^{j+1}
\frac{\pi+H_{\frac{2k-1-2j}{4}}-H_{\frac{2k-2j+1}{4}}}{2(k+1-j)}\right.\right.\nonumber\\
&-(-1)^{k}2G\left.\left.\right.+\frac{(-1)^{k}\pi\ln2}{2}\right)
-2\left[\sum_{j=1}^{k}\frac{(-1)^{j+1}}{64(k+1-j)^{2}}\right.\nonumber
\end{align}
\begin{align}
&
\times\left(-4\pi+4\psi^{(0)}\left(\frac{2k+5-2j}{4}\right)
-4\psi^{(0)}\left(\frac{2k+3-2j}{4}\right)
\right.\nonumber
\\
&-2(k+1-j)\left.\left.\left(\psi^{(1)}\left(\frac{2k+5-2j}{4}\right)\right.
-\psi^{(1)}\left(\frac{2k+3-2j}{4}\right)\right)\right)\nonumber\\
&+\frac{(-1)^{k}}{96}\left.\left.\left(3\pi^{3}+6\pi(\ln2)^{2}
\right.-192W(3)\right)\right]\nonumber\\
&+2\left[\frac{(-1)^{k}}{8}(4G-\pi\ln 2)\right.+\sum_{j=1}^{k}(-1)^{j+1}
\left.\left.\times\frac{\pi+H_{\frac{2k-1-2j}{4}}-H_{\frac{2k+1-2j}{4}}}{8(k+1-j)}
\right]\right\},
\end{align}\vspace{-0.6cm}
\begin{align}
\sum_{n=1}^{\infty}\frac{(-1)^{n+1}2h_{n}}{n(2n+2k+2)^{3}}=&
\sum_{n=0}^{\infty}\frac{(-1)^{n}t_{n}(1)}{(n+1)(2n+2k)^{3}}
=
-\frac{1}{k+1}\left\{\sum_{j=1}^{k+1}(-1)^{j+1}\right.\nonumber\\
&\times\left[\frac{1}{64(2k+3-2j)^{3}}\left(32\pi-32\psi^{(0)}
\left(\frac{k+3-j}{2}\right)\right.\right.\nonumber\\
&+32\psi^{(0)}\left(\frac{k+2-j}{2}\right)+8(2k+3-2j)
\left(\psi^{(1)}\left(\frac{k+3-j}{2}\right)\right.\nonumber\\
 &-\left.\psi^{(1)}\left(\frac{k+2-j}{2}\right)\right)
-(2k+3-2j)^{2}\left(\psi^{(2)}\left(\frac{k+3-j}{2}\right)\right.\nonumber\\
&-\left.\left.\left.\psi^{(2)}\left(\frac{2+k-j}{2}\right)\right)\right)\right]
+(-1)^{k+1}\left[\frac{1}{2}\left(\frac{11}{4}\zeta(4)\right.\right.\nonumber\\
&-\frac{7}{4}\zeta(3)\ln2+\frac{1}{2}\zeta(2)\ln^{2}2-\frac{1}{12}\ln^{4}2
\left.-2Li_{4}\left(\frac{1}{2}\right)\right)\nonumber\\
&-\left.\left.\frac{25}{128}\zeta(4)\right]\right\}
-\frac{1}{2(k+1)^{2}}\left\{-\frac{\pi^{2}}{16}+\frac{2k+1}{4(k+1)}\right.
\left(\frac{\pi^{2}}{8}\right.
\nonumber
\\
&
+\frac{(-1)^{k}\pi^{2}}{8}
-\sum_{j=1}^{k+1}(-1)^{j+1}\times\left.\left.\frac{\pi+H_{\frac{k-j}{2}}-H_{\frac{k-j+1}{2}}}{2k+3-2j}
\right)\right.\nonumber
\\
&-2\left[\sum_{j=1}^{k+1}\frac{(-1)^{j+1}}{16(2k+3-2j)^{2}}
\left(-4\pi\right.\right.+4\psi^{(0)}\left(\frac{k+3-j}{2}\right)\nonumber
\\
&-4\psi^{(0)}\left(\frac{k+2-j}{2}\right)-(2k+3-2j)
\nonumber\\
&\times\left.\psi^{(1)}\left(\frac{k+3-j}{2}\right)
+(2k+3-2j)\psi^{(1)}\left(\frac{k+2-j}{2}\right)\right)\nonumber
\\
&+\left.\frac{(-1)^{k+1}}{16}(-4G\pi+7\zeta(3))\right]\nonumber
\\&+2\left[\sum_{j=1}^{k+1}(-1)^{j+1}\frac{\pi+H_{\frac{k-j}{2}}-H_{\frac{k+1-j}{2}}}{4(2k+3-2j)}\right.
\left.\left.+\frac{(-1)^{k+1}\pi^{2}}{32}\right]\right\}.\label{EQTH8.2}
\end{align}
\end{theorem}
\begin{proof}
Using the Taylor expansion of $(\arctan (x))^{2}$ and integrating by parts, then
\begin{align*}
\int_{0}^{1}x^{a}(\arctan (x))^{2}\ln ^{2} (x)dx=&\int_{0}^{1}\sum_{n=1}^{\infty}\frac{(-1)^{(n+1)}h_{n}}{n}x^{2n+a}\ln^{2} (x)dx\\
=&\sum_{n=1}^{\infty}\frac{(-1)^{n+1}2h_{n}}{n(2n+a+1)^{3}},
\end{align*}
we can also give the following
\begin{align*}
&\int_{0}^{1}x^{a}(\arctan (x))^{2}\ln ^{2} (x)dx=\int_{0}^{1}x^{a-2}\left(\frac{\arctan (x)}{x}\right)^{2}\ln ^{2} (x)dx
\\
&=\int_{0}^{1}\frac{1}{2}\sum_{n=0}^{\infty}\frac{(-1)^{n}t_{n}(1)}{n+1}x^{2n+a-2}\ln^{2}(x)dx
=\sum_{n=0}^{\infty}\frac{(-1)^{n}t_{n}(1)}{(n+1)(2n+a-1)^{3}},\end{align*}
by \eqref{EQPTH8.1} and \eqref{EQPTH8.2}, and Theorem \ref{TH8} is proved.
\end{proof}
\begin{example}Let $k=1,2$ in \eqref{EQTH8.1} and \eqref{EQTH8.2}, then
\begin{align*}
\sum_{n=1}^{\infty}\frac{(-1)^{n+1}2h_{n}}{n(2n+3)^{3}}=&
\sum_{n=0}^{\infty}\frac{(-1)^{n}t_{n}(1)}{(n+1)(2n+1)^{3}}
=
\frac{1}{3456}\left(9\zeta\left(4,\frac{1}{4}\right)+16\pi^{2} \right.\\
&-32\pi(19+3(\ln 2)^{2}+\ln 4)-24\pi^{3}(5+\ln 64)
\\&\left.
+3072W(3)+252\pi\zeta(3)+5440-1664G-9\zeta\left(4,\frac{3}{4}\right)\right),
\\
\sum_{n=1}^{\infty}\frac{(-1)^{n+1}2h_{n}}{n(2n+4)^{3}}
=&\sum_{n=0}^{\infty}\frac{(-1)^{n}t_{n}(1)}{(n+1)(2n+2)^{3}}
=\frac{1}{69120}(80(265-54G)\pi -453\pi^{4}\\
&-160\pi ^{2}(13+9(\ln 2)^{2})+40(230+864Li_{4}\left(\frac{1}{2}\right)-27\zeta (3)
\\&+4\ln 2(-302+9(\ln 2)^{3}+189\zeta (3)))),
\\
\sum_{n=1}^{\infty}\frac{(-1)^{n+1}2h_{n}}{n(2n+5)^{3}}
=&\sum_{n=0}^{\infty}\frac{(-1)^{n}t_{n}(1)}{(n+1)(2n+3)^{3}}
=\frac{1}{864000}(305856G-345600\beta(4)
\\&+72(12\pi ^{2}+15\pi^{3}(19+20\ln 2)+3\pi(312+8\ln 2(2
\\
&+\ln 32)-175\zeta (3)))-32(30841+8640W(3))),
\\
\sum_{n=1}^{\infty}\frac{(-1)^{n+1}2h_{n}}{n(2n+6)^{3}}
=&\sum_{n=0}^{\infty}\frac{(-1)^{n}t_{n}(1)}{(n+1)(2n+4)^{3}}
=\frac{1}{2592000}(32(-15523+2250G)\pi\\
&+11325\pi^{4}+360\pi^{2}(149+100(\ln 2)^{2})-8(33787-138184\ln 2\\
&+4500(\ln 2)^{4}+108000Li_{4}\left(\frac{1}{2}\right)+900(-17+105\ln 2)\zeta(3))).
\end{align*}
\end{example}

\begin{theorem}\label{DTTH}
Let $q=0,$  then
\begin{align}
&\frac{\Gamma(p+1)}{2^{p}}\sum_{n=0}^{\infty}\frac{(-1)^{n}t_{n}(p-1)}{n+\frac{p}{2}}
\sum_{k=0}^{\infty}\frac{(-1)^{k}\Gamma(r+k)}{\Gamma(k+1)\Gamma(r)}\times\frac{1}{2k+2n+a+1}
\nonumber\\
=&c_{0}+c_{1}\pi+c_{2}\pi^{2}+\ldots+c_{p}\pi^{p}+c_{p+1}\pi^{p+1},
\end{align}
where  rational coefficients $c_{i}(i=1,2,3,\ldots)$ depend on the values of $p, r$.
\end{theorem}
\begin{proof}By Theorem \ref{CX1.1}, take $q = 0$, then
\begin{align*}
\int_{0}^{1}\frac{x^{a}\left(\frac{\arctan(x)}{x}\right)^{p}}{(1+x^{2})^{r}}dx=&
\int_{0}^{1}\frac{\Gamma(p+1)}{2^{p}}\sum_{n=0}^{\infty}\frac{(-1)^{n}t_{n}(p-1)}{n+\frac{p}{2}}
\sum_{k=0}^{\infty}\frac{(-1)^{k}\Gamma(r+k)}{\Gamma(k+1)\Gamma(r)}\\
&\times x^{2k+2n+a}dx=\frac{\Gamma(p+1)}{2^{p}}\sum_{n=0}^{\infty}\frac{(-1)^{n}t_{n}(p-1)}{n+\frac{p}{2}}
\\&\times\sum_{k=0}^{\infty}\frac{(-1)^{k}\Gamma(r+k)}{\Gamma(k+1)\Gamma(r)(2k+2n+a+1)},
\end{align*}
combining with Theorem \ref{MT4}, Theorem \ref{DTTH} is proved.
\end{proof}

\begin{example}Let $a=4, p=2, r=2$, then\vspace{-0.1cm}
\[
\frac{1}{2}\sum_{n=0}^{\infty}\frac{(-1)^{n}t_{n}(1)}{n+1}
\sum_{k=0}^{\infty}\frac{(-1)^{k}(1+k)}{2k+2n+5}
=\frac{1}{384}(48-6\pi^{2}+\pi^{3}).\]
Let $a=5, p=3, r=3$, then\vspace{-0.1cm}
\[
\frac{3}{16}\sum_{n=0}^{\infty}\frac{(-1)^{n}t_{n}(2)}{n+\frac{3}{2}}
\sum_{k=0}^{\infty}\frac{(-1)^{k}(1+k)(k+2)}{k+n+3}
=\frac{1}{8192}(-48+12\pi^{2}+\pi^{4}).\]
Let $a=6, p=4, r=2$, then\vspace{-0.1cm}
\begin{align*}\frac{3}{2}\sum_{n=0}^{\infty}\frac{(-1)^{n}t_{n}(3)}{n+2}
\sum_{k=0}^{\infty}\frac{(-1)^{k}(k+1)}{2k+2n+7}
=&\frac{1}{10240}(-3840+480\pi^{2}-10\pi^{4}+\pi^{5}).\end{align*}
Let $a=7, p=5, r=2$, then\vspace{-0.1cm}
\begin{align*}
\frac{15}{8}\sum_{n=0}^{\infty}\frac{(-1)^{n}t_{n}(4)}{n+\frac{5}{2}}
\sum_{k=0}^{\infty}\frac{(-1)^{k}(k+1)}{k+n+4}
=&\frac{1}{49152}(46080-23040\pi+960\pi^{3}-12\pi^{5}+\pi^{6}).
\end{align*}
\end{example}

\section{The sums involving $h_{n}$ and $t_{n}$ via series expansions}
For $\left(\frac{\arctan (x)}{x}\right) ^{p}$,  consider some special cases of the parameter $p$, then\vspace{-0.3cm}
\begin{align}\label{QZ1}
\left(\frac{\arctan x}{x}\right)^{2}=&\frac{1}{2}\sum_{n=0}^{\infty}\frac{(-x^{2})^{n}}{n+1}t_{n}(1),\\
\label{QZ2}
\left(\frac{\arctan x}{x}\right)^{3}=&\frac{3}{4}
\sum_{n=0}^{\infty}\frac{(-x^{2})^{n}}{n+\frac{3}{2}}t_{n}(2),\\
\label{QZ3}
\left(\frac{\arctan x}{x}\right)^{4}=&\frac{3}{2}\sum_{n=0}^{\infty}\frac{(-x^{2})^{n}}{n+2}t_{n}(3),
\\
\label{QZ4}
\left(\frac{\arctan x}{x}\right)^{5}=&\frac{15}{4}\sum_{n=0}^{\infty}\frac{(-x^{2})^{n}}{n+\frac{5}{2}}t_{n}(4),
\\
\label{QZ5}
\left(\frac{\arctan x}{x}\right)^{6}=&\frac{45}{4}\sum_{n=0}^{\infty}\frac{(-x^{2})^{n}}{n+3}t_{n}(5).
\end{align}
Setting $x=1 , x=\sqrt{3} , x=\frac{\sqrt{3}}{3} $ in \eqref{QZ1} lead to
\[
\sum_{n=0}^{\infty}\frac{(-1)^{n}}{n+1}t_{n}(1)=\frac{\pi^{2}}{8},
\sum_{n=0}^{\infty}\frac{(-3)^{n}}{n+1}t_{n}(1)=\frac{2\pi^{2}}{27},
\sum_{n=0}^{\infty}\frac{(-1)^{n}}{(n+1)3^{n}}t_{n}(1)=\frac{\pi^{2}}{6}.
\]
Differentiating \eqref{QZ1} at $x=1, x=\sqrt{3}$, then
\[\sum_{n=0}^{\infty}\frac{(-1)^{n}n}{n+1}t_{n}(1)=\frac{\pi}{4}-\frac{\pi^{2}}{8},
\sum_{n=0}^{\infty}\frac{(-1)^{n}3^{n}n}{n+1}t_{n}(1)=\frac{\sqrt{3}\pi}{18}-\frac{2\pi^{2}}{27}.\]
Let $x=1 , x=\sqrt{3} , x=\frac{\sqrt{3}}{3} $ in \eqref{QZ2}, we get
\[\sum_{n=0}^{\infty}\frac{(-1)^{n}}{n+\frac{3}{2}}t_{n}(2)=\frac{\pi^{3}}{48},
\sum_{n=0}^{\infty}\frac{(-3)^{n}}{n+\frac{3}{2}}t_{n}(2)=\frac{4\sqrt{3}\pi^{3}}{729},
\sum_{n=0}^{\infty}\frac{(-1)^{n}}{(n+\frac{3}{2})3^{n}}t_{n}(2)=\frac{\sqrt{3}\pi^{3}}{54}.\]
Differentiating \eqref{QZ2} at $x=1 , x=\sqrt{3}$, then
\[\sum_{n=0}^{\infty}\frac{(-1)^{n}n}{n+\frac{3}{2}}t_{n}(2)=\frac{\pi^2}{16}-\frac{\pi^{3}}{32},
\sum_{n=0}^{\infty}\frac{(-1)^{n}3^{n}n}{n+\frac{3}{2}}t_{n}(2)=\frac{\pi^{2}}{54}-\frac{2\sqrt{3}\pi^{3}}{243}.\]
Using $x=1 , x=\sqrt{3} , x=\frac{\sqrt{3}}{3} $ in \eqref{QZ3} yield
\[\sum_{n=0}^{\infty}\frac{(-1)^{n}}{n+2}t_{n}(3)=\frac{\pi^{4}}{384},
\sum_{n=0}^{\infty}\frac{(-3)^{n}}{n+2}t_{n}(3)=\frac{2\pi^{4}}{2187},
\sum_{n=0}^{\infty}\frac{(-1)^{n}}{(n+2)3^{n}}t_{n}(3)=\frac{\pi^{4}}{216}.\]
Differentiating \eqref{QZ3} at $x=1 , x=\sqrt{3}$, then
\[\sum_{n=0}^{\infty}\frac{(-1)^{n}n}{n+2}t_{n}(3)=\frac{\pi^3}{96}-\frac{\pi^{4}}{192},
\sum_{n=0}^{\infty}\frac{(-3)^{n}n}{n+2}t_{n}(3)=\frac{\sqrt{3}\pi^{3}}{729}-\frac{4\pi^{4}}{2187}.\]
Let $x=1 , x=\sqrt{3} , x=\frac{\sqrt{3}}{3} $ in \eqref{QZ4}, we have
\[\sum_{n=0}^{\infty}\frac{(-1)^{n}}{n+\frac{5}{2}}t_{n}(4)=\frac{\pi^{5}}{3840},
\sum_{n=0}^{\infty}\frac{(-3)^{n}}{n+\frac{5}{2}}t_{n}(4)=\frac{4\sqrt{3}\pi^{5}}{98415},\]\[
\sum_{n=0}^{\infty}\frac{(-1)^{n}}{(n+\frac{5}{2})3^{n}}t_{n}(4)=\frac{\sqrt{3}\pi^{5}}{3240}.\]
Differentiating \eqref{QZ4} at $x=1, x=\sqrt{3}$, then
\[\sum_{n=0}^{\infty}\frac{(-1)^{n}n}{n+\frac{5}{2}}t_{n}(4)=\frac{\pi^4}{768}-\frac{\pi^{5}}{1536},
\sum_{n=0}^{\infty}\frac{(-3)^{n}n}{n+\frac{5}{2}}t_{n}(4)=\frac{\pi^{4}}{4374}-\frac{2\sqrt{3}\pi^{5}}{19683}.\]
Let $x=1 , x=\sqrt{3} , x=\frac{\sqrt{3}}{3} $ in \eqref{QZ5}, \vspace{-0.2cm}
\[\sum_{n=0}^{\infty}\frac{(-1)^{n}}{n+3}t_{n}(5)=\frac{\pi^{6}}{46080},
\sum_{n=0}^{\infty}\frac{(-3)^{n}}{n+3}t_{n}(5)=\frac{4\pi^{6}}{885735},\]\vspace{-0.4cm}\[
\sum_{n=0}^{\infty}\frac{(-1)^{n}}{(n+3)3^{n}}t_{n}(5)=\frac{\pi^{6}}{19440}.\]
Differentiating \eqref{QZ5} at $x=1, x=\sqrt{3}$, we get
\[\sum_{n=0}^{\infty}\frac{(-1)^{n}n}{n+3}t_{n}(5)=\frac{(2-\pi)\pi^{5}}{15360},
\sum_{n=0}^{\infty}\frac{(-3)^{n}n}{n+3}t_{n}(5)=\frac{(9\sqrt{3}-12\pi)\pi^{5}}{885735}.\]

\begin{proposition}\label{HDS}
\begin{align}\label{HD1}
\sum_{n=1}^{\infty}\frac{(-1)^{n+1}h_{n}x^{2n-2}}{n}
=\frac{1}{2}\sum_{n=0}^{\infty}\frac{(-x^{2})^{n}t_{n}(1)}{n+1},
\\\label{HD2}
\sum_{j=1}^{\infty}\frac{(-1)^{j+1}}{2j+1}\sum_{i=1}^{j}\frac{h_{i}}{i}x^{2j-2}
=\frac{1}{4}\sum_{n=0}^{\infty}\frac{(-x^{2})^{n}t_{n}(2)}{n+\frac{3}{2}}.
\end{align}
\end{proposition}
\begin{proof}
By the Taylor expansion in \cite{SN19}
\[(\arctan (x))^{2}=\sum_{n=1}^{\infty}\frac{(-1)^{n+1}h_{n}x^{2n}}{n},(\arctan (x))^{3}=3\sum_{j=1}^{\infty}\frac{(-1)^{j+1}}{2j+1}\sum_{i=1}^{j}\frac{h_{i}}{i}x^{2j+1},\]
on the other hand, by Theorem \ref{CX1.1} and
\[\left(\frac{\arctan(x)}{x}\right)^{2}
=\frac{1}{2}\sum_{n=0}^{\infty}\frac{(-x^{2})^{n}t_{n}(1)}{n+1},
\left(\frac{\arctan(x)}{x}\right)^{3}
=\frac{3}{4}\sum_{n=0}^{\infty}\frac{(-x^{2})^{n}t_{n}(2)}{n+\frac{3}{2}},\]
Proposition \ref{HDS} is proved.\end{proof}
Let $x=1 , x=\sqrt{3} , x=\frac{\sqrt{3}}{3} $ in \eqref{HD1} and \eqref{HD2}, one has\vspace{-0.1cm}
\begin{proposition}
\begin{align*}
\sum_{n=1}^{\infty}\frac{(-1)^{n+1}h_{n}}{n}
=&\frac{1}{2}\sum_{n=0}^{\infty}\frac{(-1)^{n}}{n+1}t_{n}(1),\\
\sum_{n=1}^{\infty}\frac{(-1)^{n+1}h_{n}3^{n-1}}{n}=&\frac{1}{2}
\sum_{n=0}^{\infty}\frac{(-3)^{n}}{(n+1)}t_{n}(1),\\
\sum_{n=1}^{\infty}\frac{(-1)^{n+1}h_{n}}{3^{n-1}n}=&
\frac{1}{2}\sum_{n=0}^{\infty}\frac{(-1)^{n}}{(n+1)3^{n}}t_{n}(1),\\
\sum_{j=1}^{\infty}\frac{(-1)^{j+1}}{2j+1}\sum_{i=1}^{j}\frac{h_{i}}{i}=&
\frac{1}{4}\sum_{n=0}^{\infty}\frac{(-1)^{n}}{n+\frac{3}{2}}t_{n}(2),
\\
\sum_{j=1}^{\infty}\frac{(-1)^{j+1}}{2j+1}\sum_{i=1}^{j}\frac{h_{i}}{i}3^{j}
=&\frac{3}{4}\sum_{n=0}^{\infty}\frac{(-3)^{n}}{n+\frac{3}{2}}t_{n}(2),
\\
3\sum_{j=1}^{\infty}\frac{(-1)^{j+1}}{2j+1}\sum_{i=1}^{j}\frac{h_{i}}{i\cdot3^{j}}
=&\frac{1}{4}\sum_{n=0}^{\infty}\frac{(-1)^{n}}{(n+\frac{3}{2})3^{n}}t_{n}(2).
\end{align*}
\end{proposition}

\section{Acknowledgments}

The authors thank the referee for her/his interest in our work and for helpful comments that greatly improved
the manuscript.

\section*{Data Availibility}
Data sharing not applicable to this article as no datasets
were generated or analysed during the current study.

 \section*{Conflict of interest}

The authors have no relevant financial or non-financial
interests to disclose.




\end{document}